\documentclass[11pt,reqno]{amsart}
\usepackage{amsmath} 
\usepackage{amssymb}
\usepackage{dsfont}
\usepackage[dvips,draft,final]{graphics}
\usepackage[T1]{fontenc}
\usepackage{fancyhdr}
\usepackage{url}
\usepackage[colorlinks,linktocpage,linkcolor=blue]{hyperref}

\usepackage{color}
\usepackage{graphicx}
\newtheorem{thm}{Theorem}[section]

\newtheorem{lem}{Lemma}[section]
\newtheorem{prop}{Proposition}[section]
\newtheorem{rem}{Remark}[section]

\newtheorem{definition}{Definition}
\renewenvironment{abstract}{%
        \small
        \quotation
         \noindent {\bfseries \abstractname } }%
      {\if@twocolumn\else\endquotation\fi}
\numberwithin{equation}{section}
\newcommand{\M}{\mathrm{M}}

\newcommand{\g}{\mathrm{g}}
\newcommand{\seq}[1]{\left<#1\right>}

\newcommand{\re}{\mathfrak R}

\renewcommand{\leq}{\leqslant}
\renewcommand{\geq}{\geqslant}
\providecommand{\abs}[1]{\left\lvert#1\right\rvert}
\providecommand{\norm}[1]{\left\lVert#1\right\rVert}
\newcommand{\bel}{\begin{equation} \label}
\newcommand{\ee}{\end{equation}}
\def\R{\mathbb R}

\def\Z{\mathbb Z}

\def\la{\lambda}

\def\pa{\partial}

\def\supp{\mathrm{supp}\, }

\renewcommand{\leq}{\leqslant}
\renewcommand{\geq}{\geqslant}
\def\epsilon{\varepsilon}
\def\phi {\varphi}

\newcommand{\p}{\partial}

\title[Inverse coefficient for the wave equation]{
Inverse coefficient problem for the wave equation with nonlocal attenuation}

\author[Yavar Kian, Zouhour Rezig and Gunther Uhlmann]{ Yavar Kian$^1$, Zouhour Rezig$^2$ and Gunther Uhlmann$^3$
}
\date{}

\begin{document}

\maketitle

%%%%%%%%%%%%%%%%%%%%%%%%%%%%%%%%%%%%%%%%%%%%%%%%%
%%%%%%%%%%%%%%%%%%%%%%%%%%%%%%%%%%%%%%%%%%%%%%%%%
\begin{abstract} In this article, we investigate the inverse problem of determining time-dependent first- and zeroth-order coefficients in a wave equation with nonlocal-in-time attenuation on a Riemannian manifold from boundary measurements. This problem is motivated by imaging modalities and viscoelasticity, where such equations arise naturally, and we aim at  characterizing the properties of the underlying medium through the recovery of corresponding coefficients. Our main objective is to exploit  memory properties of the solutions, induced by the nonlocal attenuation, to establish the recovery of a general class of lower-order coefficients from boundary measurements collected over an arbitrarily small time interval. As a byproduct of our analysis, we also establish recovery results for time-dependent coefficients from measurements supported on disjoint time intervals. While such restrictions on the data are unavailable, and in general impossible, for the classical wave equation, we prove that, under standard assumptions, they become feasible due to the presence of the nonlocal-in-time attenuation. Our analysis combines techniques from differential geometry, the theory of partial differential equations with nonlocal terms, complex analysis, and the theory of nonlocal operators.

\vskip 4.5mm

%\noindent{\bf AMS Subject Classifications } 35R11, {35R30, 35B35}

\end{abstract}

\renewcommand{\thefootnote}{\fnsymbol{footnote}}
\footnotetext{\hspace*{-5mm} 
\begin{tabular}{@{}r@{}p{16cm}@{}}
$^*$
&The work of Y. Kian is supported by the French National Research Agency ANR and Hong Kong RGC Joint Research Scheme for the project IdiAnoDiff (grant ANR-24-CE40-7039). The work of G. Uhlmann is partially supported by NSF.\\
$^1$ 
& Univ Rouen Normandie, CNRS, Normandie Univ, LMRS UMR 6085, F-76000 Rouen, France  (\texttt{yavar.kian@univ-rouen.fr})\\
$^2$
& University of Tunis El ManarFaculty of Sciences of Tunis ENIT-LAMSIN, B.P. 37, 1002 Tunis, Tunisia (\texttt{zouhour.rezig@fst.utm.tn})\\
$^3$
& Department of Mathematics, University of Washington , Seattle, WA 98195-4350, USA. (\texttt{gunther@math.washington.edu})
\end{tabular}}

%%%%%%%%%%%%%%%%%%%%%%%%%%%%%%%%%%%%%%%%%%%%%%%%%
%%%%%%%%%%%%%%%%%%%%%%%%%%%%%%%%%%%%%%%%%%%%%%%%%
\section{Introduction}
\label{sec-intro}
Let $(\M,\,\g)$ be a compact connected smooth $n$-dimensional Riemannian manifold with smooth boundary $\pa \M$. We denote the Laplace-Beltrami operator associated with the Riemannian metric $\g$ by $\Delta_\g$. In local coordinates,
the metric reads $\g=(\g_{jk})_{1\leq j,k\leq n}$, and the Laplace-Beltrami operator $\Delta_\g$ is given by
$$
\Delta_\g =\frac{1}{\sqrt{\abs{g}}}\sum_{j,k=1}^n\frac{\pa}{\pa
x_j}\pa{\sqrt{\abs{g}}\,g^{jk}\frac{\pa}{\pa x_k}}.
$$
Here $(g^{jk})_{1\leq j,k\leq n}$ is the inverse of the metric $g$ and $\abs{g}=\det((g_{jk})_{1\leq j,k\leq n})$. 
We introduce $T>0$, $\alpha\in(0,1)$ and $\partial_t^\alpha$,
the Caputo fractional derivative of order $\alpha $ with respect to $t$, defined by 
$$
\partial_t^{\alpha} u(t) := 
\frac{1}{\Gamma(1-\alpha)} \int_0^t (t-s)^{-\alpha}\partial_s u(s) ds,\ u\in W^{1,1}(0,T),\ t\in(0,T),
$$
where $\Gamma$ denotes the usual Gamma function. 
Then, we consider the following initial-boundary value problem (IBVP in short)
for the  wave equation with non-local attenuation
\begin{equation}
\label{eq1}
\left\{
\begin{aligned}
& \partial_t^2 u-\Delta_\g u +a(t,x)\partial_tu+ b(t,x) \partial_t^{\alpha} u +c(t,x)u  = 0, &&\quad (t,x)\in (0,T)\times \M,\\
& u(t,x) = f(t,x),                                                                                                      &&\quad (t,x)\in (0,T)\times\pa\M,\\
& u(0,x) = 0,\quad \partial_t u(0,x) = 0,                                             &&\quad x\in\M.
\end{aligned}
\right.
\end{equation}

In the present article, we study the inverse problem of determining the damping coefficient $a$ and the potential $c$ appearing in the hyperbolic equation \eqref{eq1} from measurements given by the knowledge of solutions on the lateral boundary $(0,T)\times \partial \M$. More precisely, by exploiting the memory properties of solutions to \eqref{eq1}, we investigate the following inverse problem:

\medskip

\textbf{(IP)} Determine the pair of time-dependent coefficients $(a,c)$ from measurements restricted to $(T-\varepsilon,T)\times \partial \M$, where $\varepsilon>0$ is arbitrarily small.

\medskip

Hyperbolic equations of the form \eqref{eq1}, involving a nonlocal attenuation term in time, are commonly used to describe various physical phenomena. These include imaging modalities such as tomography in biological tissues, which are often modeled by acoustic equations with frequency-dependent attenuation terms \cite{HP,Sz} expressed through time-fractional derivatives (see \cite{CH} and \cite[Chapter 6]{TTS}), as well as wave propagation in viscoelastic media \cite{AEE,SZS}. In this context, the objective of our inverse problem is to identify the damping coefficient $a$ and the potential $c$, which characterize different properties of the underlying medium in these physical models.

The recovery of coefficients appearing in hyperbolic equations is a fundamental topic in inverse problem theory and has attracted considerable attention over the past decades. For time-independent coefficients, one of the most powerful tools for solving this class of inverse problems is the so-called boundary control method, initiated in \cite{B,BK} (see also \cite{KKL}), which relies on a combination of finite speed of propagation and unique continuation properties \cite{RoZu,Ta}. Due to limitations of unique continuation for time-dependent coefficients \cite{AB}, the boundary control method has so far been restricted to coefficients that depend analytically on time \cite{Es}. As an alternative approach, several authors have developed methods based on the construction of special classes of solutions known as geometric optics solutions. Without attempting to be exhaustive, we mention the works \cite{BZ,FIKO,FeK,Ki1,Ki2,KiOk,SY}. We also refer to recent developments addressing similar inverse problems on Lorentzian manifolds \cite{AFO1,AFO2}. Concerning inverse problems for hyperbolic equations with memory terms, we mention works devoted to source recovery \cite{AP,AP2,HKSY}, as well as studies on the identification of time-independent coefficients \cite{BuID,BDU,Dy,Y}.
We also mention the work \cite{Z}, where the author investigated the simultaneous recovery of a time-independent coefficient and a semilinear term in a nonlinear attenuated wave equation involving the fractional Laplacian.

While several of the aforementioned methods allow one to relax the geometric condition on the measurement set and consider observations restricted to a subset of the lateral boundary of the form $(0,T)\times\Gamma$, where $\Gamma$ is a nonempty open subset of $\partial \M$ (see, e.g., \cite{KKL,LO}), restrictions on the time interval of observation are generally not permitted. As a consequence, even in the case of time-independent coefficients, problem \textbf{(IP)} remains open. Indeed, it is not clear whether such a result can be achieved in view of intrinsic limitations of wave equations, such as the finite speed of propagation.

The main goal of the present article is to exploit the memory effect of the solutions to problem \eqref{eq1}, induced by the nonlocal-in-time attenuation, in order to solve problem \textbf{(IP)}. This approach is in line with recent advances in inverse problem theory, where nonlinear effects have been successfully employed to solve inverse problems that remain open for linear hyperbolic equations \cite{KLU,KLOU,UZ}. The use of attenuation in the context of hyperbolic equations is also closely related to developments in control theory, where damping mechanisms are used to stabilize solutions and to quantify their decay rates \cite{AL,Ph}.

As an application of problem \textbf{(IP)}, we also discuss the recovery of coefficients from data measured on disjoint sets. Indeed, even for hyperbolic equations with time-independent coefficients, data on disjoint sets have so far only allowed the local recovery of lower-order coefficients \cite{KKLO}. This result was  extended to global recovery under the assumption that the manifold admits a convex foliation. As another application of \textbf{(IP)}, we introduce the notion of data on sets that are disjoint in time and study the recovery of both time-dependent and time-independent coefficients from this class of measurements.

This article is organized as follows. In Section~\ref{s2}, we present our main results concerning problem~\textbf{(IP)}, stated in Theorem~\ref{t1}, together with their applications to inverse problems with data on disjoint sets in time, established in Theorems~\ref{t2} and~\ref{t3}. In Section~\ref{s3}, we recall some preliminary results on the well-posedness and regularity of solutions to~\eqref{eq1}. In Section~\ref{s4}, we construct a suitable class of geometric optics solutions to~\eqref{eq1}, which will play a central role in the proofs of our main results. Section~\ref{s5} is devoted to the proofs of Theorems~\ref{t1}--\ref{t3}. Finally, in the Appendix, we establish the memory property stated in Theorem~\ref{memory} for time-fractional derivatives, which appears to extend existing results by substantially  relaxing the regularity assumptions.

\section{Main results}\label{s2}

Let us recall the notion of simple manifolds. We say that the boundary $\pa\M$ is strictly convex if  the second fundamental form  is positive-definite for any $x \in \pa\M$.
\begin{definition}
A manifold $M$ is simple if $\pa\M$ is
strictly convex and, for any $x\in\M$, the exponential map
$\exp_x:\exp_x^{-1}(\M)\to M$ is a diffeomorphism, which means that every arbitrary two points $x, y \in\M$ can be joined by a unique geodesic.
\end{definition}

Consider  $\mathcal J$ the space defined by
\bel{J}\mathcal J=\{h\in W^{3,1}(0,T;H^{\frac92}(\pa\M))\cap W^{5,1}(0,T;H^{\frac52}(\pa\M)):\partial_t^kf(0,\cdot)\equiv0,\ k=0,\ldots,4\}.\ee
We prove in Theorem \ref{t4} that, for $a,b,c \in C^3([0,T]\times \M)$ and  $f\in\mathcal J$, problem \eqref{eq1} admits a unique solution in $u\in C([0,T];H^4(\M))\cap C^2([0,T];H^2(\M))$. Then, we can define the partial hyperbolic Dirichlet-to-Neumann (DN in short) map  \bel{pDN}\Lambda_{a,c,\epsilon}:\mathcal J\ni f\mapsto (\partial_\nu u|_{(T-\epsilon,T)\times\pa\M},\partial_\nu \Delta_\g u|_{(T-\epsilon,T)\times\pa\M})\in L^2((T-\epsilon,T)\times\pa\M)^2,\ee where $\nu$ the outward unit normal vector field on $\pa\M$ with respect to the metric $g$ and $\epsilon\in(0,T)$. The inverse problem \textbf{(IP)} can now be reformulated as the identification of the coefficients $a$ and $c$ from the knowledge of $\Lambda_{a,c,\epsilon}$ with $\epsilon$ arbitrary small.

Recall that the domain of influence $\mathcal D$ for the time space manifold $(0,T)\times \M$ is given by
$$\mathcal D:=\{(t,x)\in (0,T)\times\M :\  \textrm{dist}{(x,\pa \M)}<t<T-\textrm{dist}{(x,\pa \M)}\}.$$
It is well-known that, by finite speed of propagation, no information can be obtained about time-dependent coefficients for wave equations from any type of measurement restricted to the lateral boundary $(0,T)\times\pa\M$ and $\mathcal D$ represents the maximal set where one can, in theory, recover  such coefficients (see e.g. \cite{KiOk}). Now, for $T>2\,\textrm{Diam}(M)$, we start by fixing a subset of $\mathcal D$ given by
$$ \mathcal E_T:=\{(t,x)\in (0,T)\times\M :\  W_g(x)<t<T-W_g(x)\},$$
where $W_g(x)$ denotes the length of the longest geodesic passing through the point $x$ in $\M$.

Our first result can be stated as follows.

\begin{thm}\label{t1} For $j=1,2$, let $a_j,b,c_j \in C^3([0,T]\times \M)$ and assume that
\bel{t1a}\supp{(a_1-a_2)} \cup \supp{(c_1-c_2)} \subset \mathcal E_T,\ee
\bel{t1b} \partial_\nu^ka_1(t,x)=\partial_\nu^ka_2(t,x),\quad  k=0,1,\ (t,x)\in(0,T)\times\pa\M,\ee
\bel{t1c}\exists \epsilon_0\in(0,Diam(M)/2),\quad |b(t,x)|>0,\quad (t,x)\in [T-\epsilon_0,T]\times\pa\M.\ee
Then, for any arbitrary small $\epsilon\in(0,\epsilon_0)$, the condition
\bel{t1d}\Lambda_{a_1,c_1,\epsilon}=\Lambda_{a_2,c_2,\epsilon},\ee
implies that $a_1=a_2$ and $c_1=c_2$.
\end{thm}

To the best of our knowledge, Theorem~\ref{t1} provides the first positive answer to problem~\textbf{(IP)} for hyperbolic equations. While such results are unavailable for the classical wave equation, even in the case of time-independent coefficients, Theorem~\ref{t1} shows that, under the generic assumption~\eqref{t1c}, the memory effects exhibited by the solutions of~\eqref{eq1}, induced by the nonlocal-in-time attenuation, make this significant restriction on the measurements possible.

Our analysis combines the construction of a suitable class of solutions, known as geometric optics (GO) solutions, for~\eqref{eq1} with key properties of the nonlocal operator in time. While the construction of GO solutions for the classical wave equation on a Riemannian manifold is by now well understood (see, e.g., \cite{FIKO,FeK,SY}), in Section~\ref{s4} we provide the first extension of this theory to wave equations with a nonlocal-in-time operator and exploit these solutions to recover time-dependent coefficients. The main difficulty in this extension arises from the presence of the nonlocal term, which perturbs the classical construction of GO solutions. This construction is combined with the properties of the nonlocal operator appearing in~\eqref{eq1}, established in Theorem~\ref{memory}, which highlight the memory effects of the solutions. More precisely, in Lemma~\ref{l4}, we show how these memory effects can be exploited to reduce boundary measurements supported on an arbitrarily small time interval of the form $(T-\epsilon,T)$ to measurements on the whole interval $(0,T)$. We also note that Theorem~\ref{memory} may be of independent interest, as it appears to substantially extend the existing literature by establishing such memory properties for time-fractional derivatives under minimal regularity assumptions.

Let us also observe that the measurements considered in Theorem~\ref{t1} are described by the map $\Lambda_{a,c,\epsilon}$, which can be viewed as a restriction of the classical hyperbolic Dirichlet-to-Neumann map arising in most inverse coefficient problems for hyperbolic equations (see, e.g., \cite{B,BK,FIKO,FeK,KKL}). The additional boundary measurement $\partial_\nu\Delta_g$ may be interpreted as enhanced information on the flux, which is natural since it corresponds to the boundary trace of a classical differential operator. Moreover, such a quantity arises naturally when boundary measurements are replaced by internal measurements, as is common in many works devoted to inverse problems for nonlocal partial differential equations (see, e.g., \cite{FGKU,FKU,GRSU,GSU}).

An important feature of Theorem~\ref{t1} comes from its application to recovery of coefficients from data on disjoint sets. More precisely, fix $\epsilon_1\in (0,T-2Diam(\M)) $ and consider the sets
\bel{Jep}\mathcal J_{\epsilon_1}:=\{f\in\mathcal J:\ \textrm{supp}(f)\subset [0,T-\epsilon_1]\times\pa \M\},\ee
$$ \mathcal E_{T-\epsilon_0}:=\{(t,x)\in (0,T-\epsilon_0)\times\M :\  W_g(x)<t<T-\epsilon_0-W_g(x)\}.$$
Our second main result is an application Theorem~\ref{t1}
that can be stated as follows.

\begin{thm}\label{t2} For $j=1,2$, let $a_j,b,c_j \in C^3([0,T]\times \M)$ and assume that \eqref{t1b}-\eqref{t1c} and the following condition
\bel{t2a}\supp{(a_1-a_2)} \cup \supp{(c_1-c_2)} \subset \mathcal E_{T-\epsilon_0}\ee
are fulfilled.
Then, for any arbitrary small $\epsilon\in(0,\min(\epsilon_0,\epsilon_1)/2)$, the condition
\bel{t2b}\Lambda_{a_1,c_1,\epsilon}f=\Lambda_{a_2,c_2,\epsilon}f,\quad f\in J_{\epsilon_1}\ee
implies that $a_1=a_2$ and $c_1=c_2$.
\end{thm}

We obtain similar properties for  time-independent coefficients.

\begin{thm}\label{t3} For $j=1,2$, let $a_j,c_j \in C^3( \M)$, $b \in C^3([0,T]\times \M)$ and assume that \eqref{t1b}-\eqref{t1c} 
 and the following condition
\bel{t3a}\partial_\nu^kc_1(x)=\partial_\nu^kc_2(x),\quad   k=0,1,\ x\in\pa\M,\ee
are fulfilled.
Then, for any arbitrary small $\epsilon\in(0,\min(\epsilon_0,\epsilon_1)/2)$, the condition \eqref{t2b}
implies that $a_1=a_2$ and $c_1=c_2$.
\end{thm}

The results of Theorems~\ref{t2} and~\ref{t3} fall into the category of inverse problems with data on disjoint sets in time, since the Dirichlet excitations are applied on $[0,T-\epsilon_1]\times\partial\mathcal{M}$, whereas the measurements are restricted to $[T-\epsilon,T]\times\partial\mathcal{M}$, with $\epsilon\leq \frac{\epsilon_1}{2}$. In contrast to the existing literature on this important and challenging topic (see, e.g., \cite{KKLO,LO}), these results introduce what appears to be the first notion of data on disjoint sets in time for hyperbolic equations. As in Theorem~\ref{t1}, such a result is unavailable for the classical wave equation and relies crucially on the memory effects of the solutions to~\eqref{eq1}, which are induced by the nonlocal-in-time attenuation and characterized in Theorem~\ref{memory}.

We emphasize that, beyond their intrinsic mathematical interest, inverse problems with data on disjoint sets have important practical relevance, as they avoid any  overlap between excitations and measurements, a requirement that naturally arises in many applications. In this regard, Theorems~\ref{t2} and~\ref{t3} provide a framework for recovering the coefficients $(a,c)$ from excitations applied during the time interval $[0,T-\epsilon_1]$ and measurements collected afterwards on $[T-\epsilon,T]$, where $\epsilon\in(0,\epsilon_1/2)$ is arbitrary.

\begin{rem}
Let us also observe that, by considering a suitable class of measurements, the strategy developed in the present article can be extended to recover more general classes of first-order coefficients, similarly to the results obtained for the classical wave equation in \cite{FIKO,FeK,SY}.
\end{rem}

\section{Preliminary properties}\label{s3}

The main goal of this section is to establish the well-posedness of problem~\eqref{eq1}, together with the existence of sufficiently smooth solutions. In particular, we prove the existence of solutions with enough regularity to define the partial Dirichlet-to-Neumann map~\eqref{pDN}. We also derive suitable estimates for solutions of nonhomogeneous equations associated with~\eqref{eq1} and for the corresponding formal adjoint problem.

\subsection{Forward problem}
We state the uniqueness and existence of solutions to the initial-boundary value problem \eqref{eq1}. 
For this purpose, we consider first the following intermediate IBVP

\begin{equation}
\label{eq2}
\left\{
\begin{aligned}
& \partial_t^2 u-\Delta_\g u +a(t,x)\partial_tu+ b(t,x) \partial_t^{\alpha} u +c(t,x)u  = F(t,x), &&\quad (t,x)\in (0,T)\times \M,\\
& u(t,x) = 0,                                                                                                      &&\quad (t,x)\in (0,T)\times\pa\M,\\
& u(0,x) = 0,\quad \partial_t u(0,x) = 0,                                             &&\quad x\in\M.
\end{aligned}
\right.
\end{equation}
We show the following

\begin{lem}
\label{l1}
Let $a,b,c\in L^\infty(0,T;L^\infty(\M))$ and
$F\in L^1(0,T; L^2(\M))$. Then there exists a unique solution $u\in C([0,T]; H_0^1(\M))\cap C^1([0,T]; L^2(\M))$ to \eqref{eq2} which satisfies the following estimate 
\bel{esta}\norm{u}_{C([0,T]; H^1(\M))}+\norm{u}_{C^1([0,T]; L^2(\M))}\leq C\norm{F}_{L^1(0,T;L^2(\M))},\ee
with $C>0$ depending only on $T$, $\M$, $\alpha$, $a$, $b$, $c$.
Moreover, let $N\in \mathbb N := \{1, 2, \ldots\}$ and assume that
$a,b,c\in C^{N}([0,T]\times \M)$, \bel{l1a}F\in \bigcap_{k=0}^NW^{N-k,1}(0,T; H^k(\M))\  \textrm{with } F(0,x)=\ldots=\partial_t^{N-1}F(0,x)=0,\ x\in\M.\ee Then 
$$u\in \bigcap_{k=0}^{N+1}C^{N+1-k}([0,T];H^k(\M)).$$

\end{lem}
\begin{proof}
We consider a proof based on a fixed-point argument and a classical unique 
existence result for hyperbolic equations. We start by showing existence within the class
$C([0,T];H^1_0(\M)) \cap  C^1([0,T];L^2(\M))$ of the solution to 
\eqref{eq2}.

\textbf{Step 1:} In this step, we prove unique existence of solutions  lying in $ C([0,T];H^1_0(\M)) \cap  C^1([0,T];L^2(\M))$.
In view of \cite[Chap. 3, Theorem 8.1]{LM1}, for every  $(v_0,v_1) \in \mathcal H:=H^1_0(\M)\times L^2(\M)$, 
we can consider the $ C([0,T];H^1_0(\M)) \cap  C^1([0,T];L^2(\M))$-solution 
$v$ to the problem
\begin{equation}
\label{eq3}
\left\{
\begin{aligned}
& \partial_t^2 v-\Delta_\g v   = 0, &&\quad \textrm{in } (0,T)\times \M,\\
& v(t,x) = 0,                                                                                                      &&\quad (t,x)\in (0,T)\times\pa\M,\\
& v(0,x) = v_0(x),\quad \partial_t v(0,x) = v_1(x),                                             &&\quad x\in\M.
\end{aligned}
\right.
\end{equation}
Then, we introduce the operator 
\begin{equation}
\label{def-U0}
U_0(t) : (v_0,v_1) \to (v(t),\partial_t v(t))
\end{equation} 
and recall that $t\mapsto U_0(t)\in C([0,T];\mathcal B (\mathcal H))$. 
Here and henceforth, we denote by $\mathcal B(X,Y)$ the set of linear bounded operators from 
the Banach space $X$ to the Banach space $Y$, 
and we write $\mathcal B(X)$ instead of $\mathcal B(X,X)$. 

In light of \cite[Theorem 2.30]{KKL} (see also \cite[Theorem 2.1]{LLT}), for $a=b=c\equiv0$ and $F \in L^1(0,T;L^2(\M))$,  it is well known that \eqref{eq2} admits 
a unique solution $u$ such that 
$U_1:=(u,\partial_t u) \in  C([0,T];\mathcal H)$ reads
\begin{equation}
\label{t1d}
U_1(t)=\int_0^tU_0(t-s)(0,F(s))ds,\ t \in [0,T],
\end{equation}
and satisfies the estimate
\begin{equation}
\label{es}
\norm{U_1}_{ C([0,T];\mathcal H)}
\leq \norm{U_0}_{ C([0,T];\mathcal B(\mathcal H))}\norm{F}_{L^1(0,T;L^2(\M))}.
\end{equation} Let us also consider
\begin{equation}
\label{def-Q}A(t):=\left(\begin{array}{ll}0&0\\ c(t,\cdot)& a(t,\cdot)  \end{array}\right)\quad 
B(t,\tau):=\left(\begin{array}{ll}0&0\\ 0& b(t,\cdot)\frac{\tau^{-\alpha}}{\Gamma(1-\alpha)}  \end{array}\right),\quad t,\tau\in (0,T]
\end{equation}
and infer from Duhamel's principle that 
$u \in  C([0,T];H^1_0(\Omega)) \cap  C^1([0,T];L^2(\Omega))$ solves
\eqref{eq2} if and only if $U:=(u,\partial_t u)$ is a $ C([0,T];\mathcal H)$-solution, for $t\in[0,T]$, to 
the integral equation 
\begin{equation}
\label{FP-eq2}
U(t)=(\mathcal GU)(t):=U_1(t)-\int_0^tU_0(t-s)A(s)U(s)ds-\int_0^t\int_0^sU_0(t-s)B(s,s-\tau)U(\tau) d\tau ds.
\end{equation}
It would be enough to prove that some iterate $\mathcal G^m$ of $\mathcal G$ is a contraction mapping 
on $ C([0,T];\mathcal H)$. 
To do so, we observe that, for all $t,\tau \in (0,T]$ that $A(t),B(t,\tau) \in \mathcal B(\mathcal H)$ and that, there exists $C>0$ depending on $a,b,c$, $(\M,\g)$ and $T$, such that 
\begin{equation}
\label{FP-eq3}
\norm{A(t)}_{\mathcal B(\mathcal H)}\leq C,\quad \norm{B(t,\tau)}_{\mathcal B(\mathcal H)} \leq C \tau^{-\alpha},\quad t,\tau\in(0,T).
\end{equation}
Fix
$\mathcal K V:=\mathcal G V - U_1$ for all $V\in  C([0,T];\mathcal H)$, and observe that, for all $t\in[0,T]$, we have
\begin{eqnarray}
\norm{\mathcal K V(t)}_{\mathcal H}
&\leq & C\int_0^t\left(\norm{V(s)}_\mathcal H + \int_0^s \frac{(s-\tau)^{-\alpha}}{\Gamma(1-\alpha_1)} \norm{V(\tau)}_\mathcal H 
 d\tau\right)ds\\
&\leq & C\left( \int_0^t\norm{V(s)}_\mathcal H ds+ \int_0^t \int_\tau^t \frac{(s-\tau)^{-\alpha}}{\Gamma(1-\alpha)} \norm{V(\tau)}_\mathcal H 
ds d\tau\right) \nonumber\\
%&\leq & C \int_0^t \left(\int_\tau^t \frac{(s-\tau)^{-\alpha_1}}{\Gamma(1-\alpha_1)} ds\right) \norm{V(\tau)}_\mathcal H  %d\tau\\
& \leq & C \int_0^t\left(1+\frac{(t-\tau)^{1-\alpha}}{\Gamma(2-\alpha)}\right)\norm{V(\tau)}_\mathcal H d\tau\\
&\leq& C\left(1+\frac{T^{1-\alpha}}{\Gamma(2-\alpha)}\right)\int_0^t\norm{V(\tau)}_\mathcal H d\tau.
\label{eq-tg}
\end{eqnarray}
From this and \eqref{es}-\eqref{FP-eq2} we see that $\mathcal G$ maps 
$ C([0,T]; \mathcal H)$ into itself and, following classical arguments, one can easily check that there exists $m_0\in\mathbb N$ such that $\mathcal G^{m_0}$ is contractive on $ C([0,T];\mathcal H)$. Thus, applying the Banach fixed-point theorem, we deduce that $\mathcal G$ admits a unique fixed point $U\in  C([0,T];\mathcal H)$.
  This proves the unique existence of  a $ C([0,T];H^1_0(\M)) \cap  C^1([0,T];L^2(\M))$-solution to 
\eqref{eq2}. In addition, fixing
$$D(t)=\sup_{s\in[0,t]}\norm{U(s)}_\mathcal H,\quad t\in[0,T]$$ and applying \eqref{es}, \eqref{FP-eq2}, \eqref{FP-eq3}, we obtain
$$D(t)\leq C_0\norm{F}_{L^1(0,T;L^2(\M))}+C_1\int_0^tD(s)ds,$$
where $C_0,C_1>0$ depend only on $T$, $(\M,\g)$, $\alpha$, $a$, $b$, $c$. Then, estimate \eqref{esta} follows from a direct application of the Gronwall inequality. This completes the proof of the first statement of the lemma.

%%%%%%%%%%%%%%%% not necessary
\iffalse
This proves that \eqref{FP-eq2} has a unique solution $U\in  C([0,T];\mathcal H)$ 
fulfilling \eqref{es}. 
From this result, assuming that $U=(u,v)$ we deduce that $u\in C([0,T];H_0^1(\Omega))$,  
$\partial_t u = v\in  C([0,T];L^2(\Omega))$, and thus \eqref{sy:intro} admits a unique solution 
$u\in  C([0,T];H^1_0(\Omega))\cap  C^1([0,T];L^2(\Omega))$ satisfying
$$
\norm{u}_{ C([0,T];H^1_0(\Omega))}+\norm{u}_{ C^1([0,T];L^2(\Omega))}
\leq 
C\left(\norm{u_0}_{H^1(\Omega)}+\norm{u_1}_{L^2(\Omega)}+\norm{F}_{L^1(0,T;L^2(\Omega))}\right).
$$ 
\fi
%%%%%%%%%%%%%%%%
\textbf{Step 2.} In this step, we consider solutions with higher regularity and we prove the second claim of the theorem. Without loss of generality we may assume that $N=1$, the general case can be deduced by combining elliptic regularity with classical iterative arguments. Applying Step 1, we can consider  $w$ the solution of the problem
$$\left\{
\begin{aligned}
& \partial_t^2 w-\Delta_\g w+a\partial_tw+b\partial_t^\alpha w +cw  = G(t,x), &&\quad \textrm{in } (0,T)\times \M,\\
& w(t,x) = 0,                                                                                                      &&\quad (t,x)\in (0,T)\times\pa\M,\\
& w(0,x) = 0,\quad \partial_t w(0,x) = 0,                                             &&\quad x\in\M,
\end{aligned}
\right.$$
with
$$G(t,x)=\partial_tF(t,x)-\partial_ta\partial_tu(t,x)-\partial_tb\partial_t^\alpha u(t,x)-\partial_tcu(t,x),\quad (t,x)\in(0,T)\times \M.$$
Recalling that $G\in L^1(0,T;L^2(\M))$ and applying Step 1, we deduce that $w\in C^1([0,T];H^1_0(\M))\cap C([0,T];L^2(\M))$. Now let us fix $v\in C^2([0,T];H^1_0(\M))\cap C^1([0,T];L^2(\M))$ defined by
$$v(t,x)=\int_0^tw(s,x)ds,\quad (t,x)\in(0,T)\times \M$$
and, in light of \eqref{l1a}, observe that $v$ solves the problem
$$\left\{
\begin{aligned}
& \partial_t^2 v-\Delta_\g v+a\partial_tv+c v   = G_2(t,x), &&\quad \textrm{in } (0,T)\times \M,\\
& v(t,x) = 0,                                                                                                      &&\quad (t,x)\in (0,T)\times\pa\M,\\
& v(0,x) = 0,\quad \partial_t v(0,x) = 0,                                             &&\quad x\in\M,
\end{aligned}
\right.$$
where
$$G_1(t,x)=\partial_ta(\pa_tu-\pa_tv)(t,x)+\partial_tc(u-v)(t,x),\quad (t,x)\in(0,T)\times\M,$$
$$G_2(t,x)=F(t,x)-\int_0^t(b\partial_t^\alpha w+\partial_tb\partial_t^\alpha u+G_1)(s,x)ds,\quad (t,x)\in(0,T)\times\M.$$
In addition, recalling that $w(0,x) = 0$, $x\in\M$, we deduce that
$$\begin{aligned}\int_0^t\partial_t^\alpha w(s,\cdot)ds&=\int_0^t\int_0^s\frac{1}{\Gamma(1-\gamma)} (s-\tau)^{-\gamma} \partial_\tau w(\tau,\cdot) d\tau ds\\
&=\int_0^t\frac{1}{\Gamma(1-\gamma)} \int_0^s \tau^{-\gamma} \partial_s w(s-\tau,\cdot) d\tau ds\\
&=\int_0^t\frac{1}{\Gamma(1-\gamma)} \partial_s\left(\int_0^s (s-\tau)^{-\gamma} w(\tau,\cdot) d\tau\right)ds\\
&=\frac{1}{\Gamma(1-\gamma)}\int_0^t (t-\tau)^{-\gamma} w(\tau,\cdot) d\tau=\partial_t^\alpha v(t,\cdot), \,\ t \in(0,T).\end{aligned}$$
Therefore, $v$ solves the problem 
\bel{l1c}\left\{
\begin{aligned}
& \partial_t^2 v-\Delta_\g v+a\partial_tv+b\partial_t^\alpha v+c v   = G_3(t,x), &&\quad \textrm{in } (0,T)\times \M,\\
& v(t,x) = 0,                                                                                                      &&\quad (t,x)\in (0,T)\times\pa\M,\\
& v(0,x) = 0,\quad \partial_t v(0,x) = 0,                                             &&\quad x\in\M,
\end{aligned}
\right.\ee
where
$$G_3(t,x)=F(t,x)-\int_0^t[\partial_tb(\partial_t^\alpha u-\partial_t^\alpha v)(s,x)+G_1(s,x)]ds,\quad (t,x)\in(0,T)\times\M.$$
It is clear the $u$ solves \eqref{l1c} and the uniqueness of the solution of \eqref{l1c} implies that $u=v\in C^2([0,T];L^2(\M))\cap C^1([0,T];H^1(\M))$. Finally, we have 

$$\left\{
\begin{aligned}
& -\Delta_\g u = H(t,x), &&\quad (t,x)\in (0,T)\times \M,\\
& u(t,x) = 0,                                                                                                      &&\quad (t,x)\in (0,T)\times\pa\M,
\end{aligned}
\right.$$
with 
 $$H(t,x)=F(t,x)-\partial_t^2u(t,x)-a\partial_tu(t,x)- b \partial_t^{\alpha} u(t,x) -cu(t,x).$$ 
 Since $H\in C([0,T];L^2(\M))$, applying elliptic regularity, we deduce that $u\in C^2([0,T];L^2(\M))\cap C^1([0,T];H^1(\M))\cap C([0,T];H^2(\M))$, which completes the proof. \end{proof}

Applying Lemma \ref{l1}, we can show the following.
\begin{thm}\label{t4}Let $a,b,c \in C^3([0,T]\times \M)$ and let $f\in\mathcal J$. Then problem \eqref{eq1} admits a unique solution in $u\in C([0,T];H^4(\M))\cap C^2([0,T];H^2(\M))$.
    
\end{thm}
\begin{proof}
In light of \cite[Theorem 9.4, Chapter 1]{LM1}, we can define $L$ a bounded operator from $H^{s}(\partial \M)$ to $H^{s+\frac12}(\M)$, $s>0$, such that $Lh|_{\pa\M}=h$, $h\in H^s(\pa\M)$. Then, fixing $G$ defined by
$G(t,\cdot)=L(f(t,\cdot))$, $t\in[0,T]$,
we deduce that $G\in W^{3,1}(0,T;H^{5}( M))\cap W^{5,1}(0,T;H^3( M))$ and $G|_{(0,T)\times\pa\M}=f$.
Moreover, we have
\bel{t4b}\partial_t^kG(0,\cdot)=L(\partial_t^kh(0,\cdot))\equiv0,\quad k=0,\ldots,4.\ee
Using this property, we can  split the solution $u$ of \eqref{eq1} into two terms $u=G+v$ with $v$ solving the problem
$$\left\{
\begin{aligned}
& \partial_t^2 v-\Delta_\g v +a(t,x)\partial_tv+ b(t,x) \partial_t^{\alpha} v +c(t,x)v  = F(t,x), &&\quad (t,x)\in (0,T)\times \M,\\
& v(t,x) = 0,                                                                                                      &&\quad (t,x)\in (0,T)\times\pa\M,\\
& v(0,x) = 0,\quad \partial_t v(0,x) = 0,                                             &&\quad x\in\M,
\end{aligned}
\right.$$
with
$$F(t,x)=-(\partial_t^2 G-\Delta_\g G +a(t,x)\partial_tG+ b(t,x) \partial_t^{\alpha} G +c(t,x)G)(t,x),\quad (t,x)\in(0,T)\times \M.$$
Recalling that $G\in  W^{3,1}(0,T;H^{5}( M))\cap W^{5,1}(0,T;H^3( M))$ and applying \eqref{t4b}, we deduce that $F$ fulfills condition \eqref{l1a} with $N=3$. Thus applying Lemma \ref{l1}, we deduce that $v\in C([0,T];H^4(\M))\cap C^2([0,T];H^2(\M))$ and problem \eqref{eq1} admits a unique solution $u\in C([0,T];H^4(\M))\cap C^2([0,T];H^2(\M))$.

\end{proof}

Following Theorem \ref{t4}, for any $\epsilon\in(0,T)$, we can introduce the partial hyperbolic Dirichlet-to-Neumann map $\Lambda_{a,c,\epsilon}$ as a linear map from $\mathcal J$ to $L^2((T-\epsilon,T)\times\pa\M)^2$ defined by
$$\Lambda_{a,c,\epsilon}:\mathcal J\ni f\mapsto (\partial_\nu u|_{(T-\epsilon,T)\times\pa\M},\partial_\nu \Delta_\g u|_{(T-\epsilon,T)\times\pa\M}),$$
where $u\in C([0,T];H^4(\M))\cap C^2([0,T];H^2(\M))$ denotes the solution of \eqref{eq1}.

Similarly, let us consider the adjoint problem to \eqref{eq1}, given by 
\begin{equation}
\label{eq1*}
\left\{
\begin{aligned}
& \partial_t^2 v-\Delta_\g v -a(t,x)\partial_tv- \partial_t^{\alpha*}(b v) +(-\pa_ta(t,x)+c(t,x))v  = 0, &&\quad (t,x)\in (0,T)\times \M,\\
& v(t,x) = h(t,x),                                                                                                      &&\quad (t,x)\in (0,T)\times\pa\M,\\
& v(T,x) = 0,\quad \partial_t v(T,x) = 0,                                             &&\quad x\in\M,
\end{aligned}
\right.
\end{equation}
where
$$\pa_t^{\alpha*}(b v)(t,x)=\int_t^T\frac{(s-t)^{-\alpha}}{\Gamma(1-\alpha)}\partial_s (bv)(s,x)ds,\quad (t,x)\in(0,T)\times \M.$$
Let us first show the following useful properties of the nonlocal operators $\partial_t^\alpha$ and $\partial_t^{\alpha*}$.
\begin{lem}\label{l2}Fix $\phi_1,\phi_2\in C^1([0,T])$ such that $\phi_1(0)=\phi_2(T)=0$. Then, fixing $\widetilde{\phi_2}(t)=\phi_2(T-t)$, $t\in[0,T]$, the following properties
\bel{l2a}\partial_t^{\alpha*}\phi_2(t)=-\partial_t^\alpha\widetilde{\phi_2}(T-t),\quad t\in[0,T],\ee
\bel{l2b}\pa_t^{\alpha*}\phi_2(t)=\partial_t\left(\int_t^{T}\frac{(\tau-t)^{-\alpha}}{\Gamma(1-\alpha)} \phi_2(\tau)d\tau\right),\quad t\in[0,T],\ee
\bel{l2c}\int_0^T\partial_t^{\alpha}\phi_1(t)\phi_2(t)dt=-\int_0^T\phi_1(t)\partial_t^{\alpha*}\phi_2(t)dt\ee
 hold true.   
\end{lem}
\begin{proof}
Let us start with the first claim.
Applying the change of variable $\tau=T-s$, we find

$$\begin{aligned}
\pa_t^{\alpha*}\phi_2(t)=\int_t^T\frac{(s-t)^{-\alpha}}{\Gamma(1-\alpha)} \phi_2'(s)ds&=\int_0^{T-t}\frac{((T-\tau)-t)^{-\alpha}}{\Gamma(1-\alpha)} \phi_2'(T-\tau)d\tau\\
&=\int_0^{T-t}\frac{(T-t-\tau)^{-\alpha}}{\Gamma(1-\alpha)} (-\widetilde{\phi_2}'(\tau))d\tau\\
&=-\partial_t^\alpha\widetilde{\phi_2}(T-t).\end{aligned}$$
This proves \eqref{l1a}. Moreover, recalling that $\phi_2(T)=0$, we get
$$\begin{aligned}\partial_t\left(\int_t^{T}\frac{(\tau-t)^{-\alpha}}{\Gamma(1-\alpha)} \phi_2(\tau)d\tau\right)&=\partial_t\left(\int_0^{T-t}\frac{s^{-\alpha}}{\Gamma(1-\alpha)} \phi_2(t+s)ds\right)\\
&=-\frac{(T-t)^{-\alpha}}{\Gamma(1-\alpha)} \phi_2(T)+\int_0^{T-t}\frac{s^{-\alpha}}{\Gamma(1-\alpha)} \phi_2'(t+s)ds\\
&=\int_t^{T}\frac{(\tau-t)^{-\alpha}}{\Gamma(1-\alpha)} \phi_2'(\tau)d\tau=\pa_t^{\alpha*}\phi_2(t).\end{aligned}$$
This identity proves \eqref{l2b}.

Now, let us consider \eqref{l2c}. Integrating by parts, using the fact that $\phi_1(0)=\phi_2(T)=0$ and applying Fubini's theorem, we obtain
$$\begin{aligned}\int_0^T\partial_t^{\alpha}\phi_1(t)\phi_2(t)dt&=\frac{1}{\Gamma(1-\alpha)}\int_0^T \left(\int_0^t (t-s)^{-\alpha}\phi_1'(s) ds\right)\phi_2(t)dt\\
&=\frac{1}{\Gamma(1-\alpha)}\int_0^T \partial_t\left(\int_0^t s^{-\alpha}\phi_1(t-s) ds\right)\phi_2(t)dt\\
&=-\frac{1}{\Gamma(1-\alpha)}\int_0^T \int_0^t (t-s)^{-\alpha}\phi_1(s) \phi_2'(t)dsdt\\
&=-\frac{1}{\Gamma(1-\alpha)}\int_0^T \phi_1(s) \left(\int_s^T (t-s)^{-\alpha} \phi_2'(t)dt\right)ds=\int_0^T\phi_1(t)\partial_t^{\alpha*}\phi_2(t)dt.\end{aligned}$$
Therefore, we obtain \eqref{l2c}. This completes the proof of the lemma.
\end{proof}

Consider also the adjoint problem to \eqref{eq2} given by
\begin{equation}
\label{eq2*}
\left\{
\begin{aligned}
& \partial_t^2 v-\Delta_\g v -a(t,x)\partial_tv- \partial_t^{\alpha*}(bv) +(-\pa_ta(t,x)+c(t,x))v  = G(t,x), &&\quad (t,x)\in (0,T)\times \M,\\
& v(t,x) = 0,                                                                                                      &&\quad (t,x)\in (0,T)\times\pa\M,\\
& v(T,x) = 0,\quad \partial_t v(T,x) = 0,                                             &&\quad x\in\M.
\end{aligned}
\right.
\end{equation}
Repeating the argumentation of Lemma \ref{l1} and applying Lemma \ref{l2}, we can prove the following.
\begin{lem}
\label{l3}
Let $a,b,\partial_tb, c\in L^\infty(0,T;L^\infty(\M))$ and
$G\in L^1(0,T; L^2(\M))$. Then there exists a unique solution $v\in \mathcal{C}([0,T]; H_0^1(\M))\cap \mathcal{C}^1([0,T]; L^2(\M))$ to \eqref{eq2*} which satisfies the following estimate \bel{estb}\norm{v}_{C([0,T]; H^1(\M))}+\norm{v}_{C^1([0,T]; L^2(\M))}\leq C\norm{G}_{L^1(0,T;L^2(\M))},\ee
$C>0$ depending only on $T$, $(\M,\g)$, $\alpha$, $a$, $b$, $c$. Moreover, let $N\in \mathbb N$ and assume that
$a,b,\partial_tb,c\in C^{N}([0,T]\times \M)$ and \eqref{l1a} is fulfilled.  Then 
$$v\in \bigcap_{k=0}^{N+1}C^{N+1-k}([0,T];H^k(\M)).$$
\end{lem}
Using Lemma \ref{l1}, \ref{l3}, we consider also the following.
\begin{prop}\label{p1}Let $a,c \in C^1([0,T]\times \M)$, $b\in C^2([0,T]\times \M)$ and let $f,h\in H^1_0((0,T)\times\pa\M)$. Then problem \eqref{eq1} $($resp. \eqref{eq1*}$)$ admits a unique solution  $u\in C([0,T];H^1_0(\M))\cap  C^1([0,T];L^2(\M))$ $($resp. $v\in C([0,T];H^1_0(\M))\cap  C^1([0,T];L^2(\M))$$)$ with $\partial_\nu u\in L^2((0,T)\times \pa\M)$ $($resp. $\partial_\nu v\in L^2((0,T)\times \pa\M)$$)$. Moreover, the following estimates holds true
\bel{p1a}\norm{u}_{C([0,T]; H^1(\M))}+\norm{u}_{C^1([0,T]; L^2(\M))}+\norm{\partial_\nu u}_{L^2((0,T)\times\pa\M)}\leq C\norm{f}_{H^1((0,T)\times \pa\M))},\ee
\bel{p1b}\norm{v}_{C([0,T]; H^1(\M))}+\norm{v}_{C^1([0,T]; L^2(\M))}+\norm{\partial_\nu v}_{L^2((0,T)\times\pa\M)}\leq C\norm{h}_{H^1((0,T)\times \pa\M))},\ee
 with $C>0$ depending only on $T$, $M$, $\alpha$, $a$, $b$, $c$.   
\end{prop}
\begin{proof}
We will only consider the result for problem \eqref{eq1}, the proof being similar for \eqref{eq1*}. In view of
\cite[Theorem 2.30]{KKL} (see also \cite[Theorem 2.1]{LLT}), there exists a unique $C([0,T];H^1_0(\M))\cap  C^1([0,T];L^2(\M))$-solution  $w$ to the IBVP
$$\left\{
\begin{aligned}
& \partial_t^2 w-\Delta_\g w   = 0, &&\quad (t,x)\in (0,T)\times \M,\\
& w(t,x) = f(t,x),                                                                                                      &&\quad (t,x)\in (0,T)\times\pa\M,\\
& w(0,x) = 0,\quad \partial_t w(0,x) = 0,                                             &&\quad x\in\M.
\end{aligned}
\right.$$
Moreover, $\partial_\nu w\in L^2((0,T)\times\pa\M)$ and we have
\bel{p1c}\norm{w}_{C([0,T]; H^1(\M))}+\norm{w}_{C^1([0,T]; L^2(\M))}+\norm{\partial_\nu w}_{L^2((0,T)\times\pa\M)}\leq C\norm{f}_{H^1((0,T)\times \pa\M))},\ee
with $C>0$ depending on $T$ and $M$.
Thus, we can split the solution $u$ of \eqref{eq1} into two terms $u=y+w$, where $y$ solves the IBVP
\bel{p1d}
\left\{
\begin{aligned}
& \partial_t^2 y-\Delta_\g y +a(t,x)\partial_ty+ b(t,x) \partial_t^{\alpha} y +c(t,x)y  = F_1(t,x), &&\quad (t,x)\in (0,T)\times \M,\\
& y(t,x) = 0,                                                                                                      &&\quad (t,x)\in (0,T)\times\pa\M,\\
& y(0,x) = 0,\quad \partial_t y(0,x) = 0,                                             &&\quad x\in\M.
\end{aligned}
\right.
\ee
with $F_1=-a\partial_tw- b \partial_t^{\alpha} w -cw$. It is clear that $F_1\in L^1(0,T;L^2(\M))$ and, applying estimate \eqref{p1c}, we obtain 
$$\norm{F_1}_{L^1(0,T;L^2(\M))}\leq C(\norm{w}_{C([0,T]; H^1(\M))}+\norm{w}_{C^1([0,T]; L^2(\M))})\leq C\norm{f}_{H^1((0,T)\times \pa\M))},$$
with $C>0$ depending only on $T$, $M$, $\alpha$, $a$, $b$, $c$. Therefore, in light of Lemma \ref{l1}, problem \eqref{p1d} admits a unique solution $y\in C([0,T];H^1_0(\M))\cap  C^1([0,T];L^2(\M))$ satisfying the estimate 
\bel{p1f}
\norm{y}_{C([0,T]; H^1(\M))}+\norm{y}_{C^1([0,T]; L^2(\M))}\leq C\norm{F_1}_{L^1(0,T;L^2(\M))}\leq C\norm{f}_{H^1((0,T)\times \pa\M))},\ee
with $C>0$ depending only on $T$, $M$, $\alpha$, $a$, $b$, $c$. Finally, fixing $F_2=-a\partial_ty- b \partial_t^{\alpha} y -cy$, we observe that $F_2\in L^1(0,T;L^2(\M))$ and $y$ solves the IBVP
$$\left\{
\begin{aligned}
& \partial_t^2 y-\Delta_\g y   = F_1+F_2, &&\quad (t,x)\in (0,T)\times \M,\\
& y(t,x) = 0,                                                                                                      &&\quad (t,x)\in (0,T)\times\pa\M,\\
& y(0,x) = 0,\quad \partial_t y(0,x) = 0,                                             &&\quad x\in\M.
\end{aligned}
\right.$$
Thus, applying \cite[Theorem 2.30]{KKL}, we deduce that $\partial_\nu y\in L^2((0,T)\times\pa\M)$ and
$$\norm{\partial_\nu y}_{L^2((0,T)\times\pa\M)}\leq C\norm{F_1+F_2}_{L^1(0,T;L^2(\M))}\leq C(\norm{F_1}_{L^1(0,T;L^2(\M))}+\norm{y}_{C^1([0,T]; L^2(\M))}). $$
In view of \eqref{p1f}, we obtain 
\bel{p1g}\norm{\partial_\nu y}_{L^2((0,T)\times\pa\M)}\leq C\norm{f}_{H^1((0,T)\times \pa\M))},\ee
with $C>0$ depending only on $T$, $M$, $\alpha$, $a$, $b$, $c$. Combining all the above mentioned properties, we deduce that $u=w+y$ is the unique $ C([0,T];H^1_0(\M))\cap  C^1([0,T];L^2(\M))$-solution of \eqref{eq1} and it satisfies $\partial_\nu u=\partial_\nu w+\partial_\nu y\in L^2((0,T)\times\pa\M)$. Finally, combining \eqref{p1c}, \eqref{p1f} and \eqref{p1g}, we deduce that \eqref{p1a} holds true. This completes the proof of the proposition.

\end{proof}

Following Proposition \ref{p1}, we can define the hyperbolic DN map $\mathcal N_{a,c}$, associated with \eqref{eq1}, as a bounded linear map from $H^1_0((0,T)\times\pa\M)$ to $L^2((0,T)\times\pa\M)$ defined by
$$\mathcal N_{a,c}:H^1_0((0,T)\times\pa\M)\ni f\mapsto \partial_\nu u|_{(0,T)\times\pa\M},$$
where $u$ is the unique $ C([0,T];H^1_0(\M))\cap  C^1([0,T];L^2(\M))$-solution of \eqref{eq1}.
\subsection{Intermediate results}

In this subsection we consider estimates of solutions of \eqref{eq2} and \eqref{eq2*} that will be useful for the construction of geometric optics in the next section. These results can be stated as follows.

\begin{prop}
\label{p2}
Let $a,b \in C^1([0,T];C( M))$, $c\in C([0,T]\times \M)$,  $F \in L^2((0,T)\times \M))$ and suppose $u$ $($resp. $v$$)$ is the unique solution to \eqref{eq2} $($resp. \eqref{eq2*}$)$. Then the following estimate holds:
\bel{p2a}\begin{aligned}
&\|u\|_{C([0,T];L^2(\M))} \leq C  \norm{\int_0^t F(s) \,ds}_{L^{2}((0,T)\times \M)}\\
&\left(\textrm{resp. }\|v\|_{C([0,T];L^2(\M))} \leq C  \norm{\int_{T-t}^T G(s) \,ds}_{L^{2}((0,T)\times \M)}\right),\end{aligned}
\ee
with $C>0$ independent of $F$ and $G$.
\end{prop}
\begin{proof} We divide the proof into two steps starting with estimate \eqref{p2a} for the solution of \eqref{eq2}, then considering the solution of \eqref{eq2*}.

\textbf{Step 1.} We set $w(t,x):=\int_0^t u(s,x)\,ds$, $(t,x)\in[0,T]\times\M$, and note that $w$ solves the IBVP
\begin{equation}
\label{eq10}
\left\{
\begin{aligned}
& \partial_t^2 w-\Delta_\g w   = H, &&\quad \textrm{in } (0,T)\times \M,\\
& w(t,x) = 0,                                                                                                      &&\quad (t,x)\in (0,T)\times\pa\M,\\
& w(0,x) = 0,\quad \partial_t w(0,x) = 0,                                             &&\quad x\in\M,
\end{aligned}
\right.
\end{equation}
with
$$H(t,x):=-\int_0^t[a\partial_tu(s,x)+ b \partial_t^{\alpha} u(s,x) +cu(s,x)]ds+\int_0^tF(s,x)\,ds,\quad (t,x)\in(0,T)\times\M.$$
Recalling that $u\in \mathcal{C}([0,T]; H_0^1(\M))\cap \mathcal{C}^1([0,T]; L^2(\M))$, we get $w\in  C^2([0,T];L^2(\M))\cap  C^1([0,T];H^1_0(\M))$. In addition, since $H\in C([0,T];L^2(\M))$ and that $w$ solves the elliptic boundary value problem
$$\left\{ \begin{array}{rcll} -\Delta_\g  w = H+\partial_t^2w,&  \textrm{in } (0,T)\times \M ,\\
w(t,x)=0, & (t,x) \in (0,T)\times\partial \M,& \end{array}\right.$$
by the elliptic regularity, we obtain  $w\in C([0,T];H^2(\M))$ and it follows that $w\in  C^2([0,T];L^2(\M))\cap  C^1([0,T];H^1_0(\M))\cap C([0,T];H^2(\M))$. We define the energy $E(t)$ at time $t$ associated with $w$ and given by
\bel{p2b} E(t):=\int_\M \left(|\partial_tw|^2(t,x)+|\nabla_gw|_\g ^2(t,x)\right)\,dV_\g (x)\geq \int_\M |u(t,x)|^2\,dV_\g (x),\quad t\in[0,T].\ee
Multiplying \eqref{eq10} by $\overline{\partial_tw}$ and taking the real part, we find
\bel{p2c}\begin{aligned}E(t)&= 2\re \left(\int_0^t\int_\M \left(\pa_t^2w(s,x)-\Delta_\g w(s,x)\right)\overline{\partial_tw(s,x)}\,dV_\g (x)\,ds\right)\\
&=-2\re \left(\int_0^t\int_\M \left(\int_0^sa(\tau,x)\partial_tu(\tau,x)\,d\tau\right)\overline{\partial_tw(s,x)}\,dV_\g (x)\,ds\right)\\
&\ \ \ -2\re \left(\int_0^t\int_\M \left(\int_0^sb(\tau,x)\partial_t^\alpha u(\tau,x)\,d\tau\right)\overline{\partial_tw(s,x)}\,dV_\g (x)\,ds\right)\\
&\ \ \ -2\re \left(\int_0^t\int_\M \left(\int_0^sc(\tau,x)u(\tau,x)\,d\tau\right)\overline{\partial_tw(s,x)}\,dV_\g (x)\,ds\right)\\
&\ \ \ +2\re \left(\int_0^t\int_\M \left(\int_0^sF(\tau,x)\,d\tau\right)\overline{\partial_tw(s,x)}\,dV_\g (x)\,ds\right),\quad t\in[0,T].\end{aligned}\ee
Now let us fix $D$ defined on $[0,T]$ by 
$$D(t)=\sup_{s\in[0,t]}E(s),\quad t\in [0,T].$$
Applying  Cauchy-Schwarz inequality, for all $t\in[0,T]$, we get
\bel{p2d}\begin{aligned}&\abs{\int_0^t\int_\M \left(\int_0^sc(\tau,x)u(\tau,x)\,d\tau\right)\overline{\partial_tw(s,x)}\,dV_\g (x)\,ds}\\
&\leq T^{\frac 12} \norm{c}_{ L^{\infty}(0,T;L^\infty( M))}\int_0^t\left(\int_0^sE(\tau)d\tau\right)^{\frac12} D(s)^{\frac12}ds\\
&\leq T\norm{c}_{ L^{\infty}(0,T;L^\infty( M))}\int_0^tD(s)ds,\end{aligned}\ee

\bel{p2e}\begin{aligned}\abs{\int_0^t\int_\M \left(\int_0^sF(\tau,x)\,d\tau\right)\overline{\partial_tw(s,x)}\,dV_\g (x)\,ds}
&\leq \norm{\int_0^tF(s,\cdot)ds}_{L^2((0,T)\times \M))}\left(\int_0^tE(s)ds\right)^{\frac{1}{2}}\\
&\leq \norm{\int_0^tF(s,\cdot)ds}_{L^2((0,T)\times \M))}^2+\int_0^t D(s)\,ds.\end{aligned}\ee Moreover, recalling that $ a\in C^1([0,T];C(\M))$, $u\in  C^1([0,T];L^2(\M))$, with $u(0,\cdot)\equiv0$,  and integrating by parts, we get
$$\begin{aligned} &\int_0^t\int_\M \left(\int_0^sa\partial_tu(\tau,x)d\tau\right)\overline{\partial_tw(s,x)}\,dV_\g (x)\,ds\\
&=-\int_0^t\int_\M \left(\int_0^s\partial_tau(\tau,x)d\tau\right)\overline{\partial_tw(s,x)}\,dV_\g (x)\,ds+\int_0^t\int_\M a|\partial_tw(s,x)|^2\,dV_\g (x)\,ds.\end{aligned}$$
Thus, repeating the above argumentation, we find
\bel{p2f}
\begin{aligned}
&\abs{\int_0^t\int_\M \left(\int_0^sa\partial_tu(\tau,x)d\tau\right)\overline{\partial_tw(s,x)}\,dV_\g (x)\,ds} \leq C\int_0^t D(\tau)\,d\tau,\quad t\in[0,T],
\end{aligned}
\ee
with $C$ depending on $T$, $M$ and $a$. 
Recalling that $u(0,\cdot)\equiv0$, we deduce that, for all $(t,x)\in(0,T)\times \M$, we have
$$\partial_t^\alpha u(t,x)=\partial_t\left(\int_0^t\frac{\tau^{-\alpha}}{\Gamma(1-\alpha)}u(t-\tau,x)d\tau\right)=\partial_t\left(\int_0^t\frac{(t-\tau)^{-\alpha}}{\Gamma(1-\alpha)}u(\tau,x)d\tau\right).$$
Therefore, we obtain
$$\begin{aligned} &\int_0^t\int_\M \left(\int_0^sb\partial_t^\alpha u(\tau,x)d\tau\right)\overline{\partial_tw(s,x)}\,dV_\g (x)\,ds\\
&=-\int_0^t\int_\M \left(\int_0^s\partial_tb\left(\int_0^\tau \frac{(\tau-\tau_1)^{-\alpha}}{\Gamma(1-\alpha)}u(\tau_1,x)d\tau_1\right)d\tau\right)\overline{\partial_tw(s,x)}\,dV_\g (x)\,ds\\
&\ \ \ \ +\int_0^t\int_\M b\left(\int_0^s \frac{(s-\tau)^{-\alpha}}{\Gamma(1-\alpha)}u(\tau,x)d\tau\right)\overline{\partial_tw(s,x)}\,dV_\g (x)\,ds.\end{aligned}$$
It follows that

\bel{p2g}
\begin{aligned}
&\abs{\int_0^t\int_\M \left(\int_0^sb\partial_t^\alpha u(\tau,x)d\tau\right)\overline{\partial_tw(s,x)}\,dV_\g (x)\,ds} \\
&\leq C\int_0^t\left(\int_0^s\int_0^\tau (\tau-\tau_1)^{-\alpha}d\tau_1d\tau\right)D(s)ds+\int_0^t\left(\int_0^s (s-\tau)^{-\alpha}d\tau\right)D(s)ds\\
&\leq C\left(\int_0^ts^{2-\alpha}D(s)ds+\int_0^ts^{1-\alpha}D(s)ds\right)\leq C(T^{2-\alpha}+T^{1-\alpha})\int_0^tD(s)ds,\quad t\in[0,T],
\end{aligned}
\ee
with $C$ depending on $T$, $M$ and $b$. Combining the identity \eqref{p2c} with the estimates \eqref{p2e}-\eqref{p2g}, we obtain
$$D(t)\leq \norm{\int_0^tF(s,\cdot)ds}_{L^2((0,T)\times \M))}^2+C\int_0^t D(\tau)\,d\tau,\quad t\in[0,T],$$
with $C$ depending on $T$, $M$, $a$, $b$ and $c$. Then applying  Gronwall inequality yields
$$D(t)\leq \norm{\int_0^tF(s,\cdot)ds}_{L^2((0,T)\times \M))}e^{Ct},\quad t\in[0,T].$$
Choosing $t=T$, we get
$$\sup_{t\in[0,T]}E(t)=D(T)\leq \norm{\int_0^tF(s,\cdot)ds}_{L^2((0,T)\times \M))}e^{CT}$$
which clearly implies \eqref{p2a}.

\textbf{Step 2.} Let us consider \eqref{p2a} for the solution $v$ of \eqref{eq2*}. For this purpose, fix $\tilde{v}, \tilde{b}$ defined by
$$\tilde{v}(t,x)=v(T-t,x),\quad \tilde{b}(t,x)=b(T-t,x),\quad (t,x)\in[0,T]\times \M.$$
Recalling that $v\in C^1([0,T];L^2(\M))$, with $v(T,\cdot)\equiv0$, and applying Lemma \ref{l2}, we observe that
$$\partial_t^{\alpha*}(bv)(t,x)=\partial_t^{\alpha}(\tilde{b}\tilde{v})(T-t,x),\quad \in[0,T]\times \M.$$
Then, $\tilde{v}$ solves the problem
$$\left\{
\begin{aligned}
& \partial_t^2 \tilde{v}-\Delta_\g \tilde{v} +a(T-t,x)\partial_t\tilde{v}+  \partial_t^{\alpha}(\tilde{b}\tilde{v})  +c(T-t,x)v  = G(T-t,x), &&\quad (t,x)\in (0,T)\times \M,\\
& \tilde{v}(t,x) = 0,                                                                                                      &&\quad (t,x)\in (0,T)\times\pa\M,\\
& \tilde{v}(0,x) = 0,\quad \partial_t \tilde{v}(0,x) = 0,                                             &&\quad x\in\M.
\end{aligned}
\right.$$
Therefore, repeating the argumentation of Step 1, we find
$$\norm{\tilde{v}}_{C([0,T];L^2(\M))}\leq C\norm{\int_0^tG(T-s,\cdot)ds}_{L^2((0,T)\times \M))}=C\norm{\int_{T-t}^TG(s,\cdot)ds}_{L^2((0,T)\times \M))},$$
with $C$ depending on $T$, $M$, $a$, $b$ and $c$. From this last estimate, we deduce easily \eqref{p2a} for the solution $v$ of \eqref{eq2*}.

\end{proof}

\section{Construction of geometric optics solutions}\label{s4}
\setcounter{equation}{0}
In this section, we aim to construct geometric optics solutions  to the equation \eqref{eq1} and its adjoint equation
\eqref{eq1*}. More precisely, we consider solutions $u\in C([0,T];H^1(\M))\cap  C^1([0,T];L^2(\M))$,  with $\partial_\nu u\in L^2((0,T)\times \pa \M)$, of the equation
\begin{equation}\label{wave eq}
\partial ^2_tu-\Delta_\g u+a(t,x) \partial_tu +b(t,x) \partial_t^\alpha u+c(t,x)u=0\quad \textrm{in}\quad (0,T)\times\M,
\end{equation}
with the initial condition
\begin{equation}\label{initial cond u}
u(0,x)=\partial_tu(0,x)=0,\quad x\in\M,
\end{equation}
 of the form
\begin{equation}\label{GO}
u(t,x)=\beta(t,x)\chi_a(t,x)e^{i\lambda \psi(t,x)}+R_\lambda(t,x),
\end{equation}
where $\lambda > 0$ is an arbitrary large parameter  and $R_{\lambda}$ is a term that decays with respect to the parameter $\lambda$. Similarly, we consider solutions $v\in C([0,T];H^1(\M))\cap  C^1([0,T];L^2(\M))$,  with $\partial_\nu v\in L^2((0,T)\times \pa \M)$, of the equation
\begin{equation}\label{wave eq*}
\partial ^2_tv-\Delta_\g v -a(t,x) \partial_tv - \partial_t^{\alpha*}(bv) +c(t,x)v=0\quad \textrm{in}\quad (0,T)\times\M,
\end{equation}
with the final condition
\begin{equation}\label{final cond v}
v(T,x)=\partial_tv(T,x)=0,\quad x\in \M,
\end{equation}
 of the form
\begin{equation}\label{GO*}
v(t,x)=\beta(t,x)\chi_{-a}(t,x)e^{-i\lambda \psi(t,x)}+R_\lambda^*(t,x).
\end{equation}

Let $a,b\in C^3([0,T]\times\M)$,  and $c\in C([0,T]\times\M)$. We assume that there exist $\psi\in C^\infty([0,T]\times\M)$,  $\beta\in C^\infty([0,T]\times M)$ and  $\chi_{a}\in W^{3,\infty}(0,T;L^\infty(\M))\cap W^{1,\infty}(0,T;W^{2,\infty}(\M))$ solving respectively the following equations
\begin{equation}\label{eikonal}
(\partial_t\psi)^2-|\nabla_g\psi|_g=0,\quad  (t,x)\in [0,T]\times\M,
\end{equation}
\begin{equation}\label{transport1}
\p_t \beta\partial_t\psi-\langle d\psi,d\beta \rangle_\g-\frac{1}{2} (\Delta_g \psi)\beta=0,\quad  t\in[0,T],\, x\in\M,
\end{equation}
\begin{equation}\label{transport2}
\p_t \chi_a\partial_t\psi-\langle d\psi,d\chi_a\rangle_\g+\frac{a}{2}\chi_a=0,\quad  t\in[0,T],\, x\in\M.
\end{equation}
In addition,  we impose to $\beta$ the following conditions
\bel{in-fa}\beta(0,x)=\beta(T,x)=0,\quad x\in \M.\ee
We assume also that the remainder terms $R_\lambda$ and $R_\lambda^*$ are lying in $C([0,T];H^1_0(\M))\cap  C^1([0,T];L^2(\M))$,  with $\partial_\nu R_\lambda, \partial_\nu R_\lambda^*\in L^2((0,T)\times \pa \M)$, and they satisfy the decay property
\bel{decay}\norm{R_\lambda}_{C([0,T];L^2(\M))}+\lambda^{-1}\norm{R_\lambda}_{C^1([0,T];L^2(\M))}+\norm{R_\lambda^*}_{C([0,T];L^2(\M))}+\lambda^{-1}\norm{R_\lambda^*}_{C^1([0,T];L^2(\M))}\leq C\lambda^{\alpha-1},\ee
with $C>0$ independent of $\lambda>1$.

\subsection{Principal parts of the geometric optics solutions}\label{section GOS}
We give an explicit representation of the term $\psi$,  $\beta$ and  $\chi_{a}$ as solutions of equations \eqref{eikonal}, \eqref{transport1} and \eqref{transport2}. Following the approach of \cite[Section 3]{FeK}, we  consider construction of these terms adjusted to any arbitrary chosen maximal null geodesic $\tilde{\gamma} \subset  \mathcal D $  whose projection onto the Riemannian manifold $\M$ will be denoted by $\gamma$.

We start by extending the simple manifold $(\M,\g)$ into a simple manifold $(\M_1,\g_1)$ in such a way that $\M$ is contained into the interior of $\M_1$ and $\g_1|_{T\M\times T\M}=\g$. For $x\in \M_1$ and $\theta\in T_x\M_1$, denote by $\gamma_{x,\theta}$ the unique geodesic starting from $x$ and directed by $\theta$, which is defined on the maximal interval $[\tau_-(x,\theta),\tau_+(x,\theta)]$, with $\gamma_{x,\theta}(\tau_\pm(x,\theta))\in\pa\M_1$.  
Recall that the sphere bundle and co-sphere bundle of $\M_1$ are respectively given by
$$
S\M_1=\{(x,\theta)\in T\M_1;\,\abs{\theta}_{\g_1(x)}=1\}, \quad
S^*\M_1=\{(x,p)\in T^*\M_1;\,\abs{p}_{\g_1(x)}=1\}.
$$
Introduce now the submanifolds of inner and outer vectors of $S\M_1$
\begin{equation}\label{2.8}
\p_{\pm}S\M_1 =\{(x,\theta)\in S\M_1,\, x \in \p \M_1,\, \pm\seq{\theta,\textbf{n}(x)}_{\g_1(x)}< 0\},
\end{equation}
where $\textbf{n}$ is the unit outer normal vector field on $\p\M_1$.

%Now pick $y\in \pa \M_1$ and consider the polar normal coordinates $(r,\theta)$ on $\M_1$ given by $x=\exp_y(r\theta)$, where   $\theta\in S_{y}^+\M_1=\{\theta\in T_{y}\M_1,\,\,\abs{\theta}=1,\,\,\seq{\theta,\textbf{n}(y)}_{\g_1(y)}<0\}$ and $r\in(0,\tau_+(y,\theta))$.  According to the Gauss lemma (see e.g. \cite[Chapter 9, Lemma 15]{Sp}), in these coordinates the metric takes the form $\tilde{g}(r,\theta)=dr^2+g_0(r,\theta)$ with $g_0(r,\theta)$ a metric on $S_y\M_1$ that depends smoothly on $r$. We fix also $m(r,\theta)=det(g_0(r,\theta))$, $\theta\in S_{y}^+\M_1$, $r\in(0,\tau_+(y,\theta))$.

 Since $\M \subset \M_1^{int}$, where $\M_1^{int}$ denotes the interior of $\M_1$, we extend $\gamma$ uniquely to a maximal unit-speed geodesic in  the simple manifold $(\M_1,\g_1)$ and we fix a point $y \in \partial \M_1 \cap \gamma$. Then, we consider the polar normal coordinates $(r,\theta)$ on $\M_1$ given by $x=\exp_y(r\theta)$, where   $\theta\in S_{y}^+\M_1=\{\theta\in T_{y}\M_1,\,\,\abs{\theta}=1,\,\,\seq{\theta,\textbf{n}(y)}_{\g_1(y)}<0\}$ and $r\in(0,\tau_+(y,\theta))$.  According to the Gauss lemma (see e.g. \cite[Chapter 9, Lemma 15]{Sp}), in these coordinates the metric takes the form $\tilde{g}(r,\theta)=dr^2+g_0(r,\theta)$, with $g_0(r,\theta)$ a metric on $S_y\M_1$ that depends smoothly on $r$. We fix also $m(r,\theta)=det(g_0(r,\theta))$, $\theta\in S_{y}^+\M_1$, $r\in(0,\tau_+(y,\theta))$. It is well known that there exists a unique inward unit vector $ \theta \in \partial_+S\M_1$ such that $\gamma= \gamma_{y,\theta}(t),$ defined in $\M_1$ on its maximal interval $[0, \tau_+(y, \theta)]$. The extended null geodesic $\tilde{\gamma}_1$ can then be parametrized by  
$$\tilde{\gamma}(t; s, y, \theta) = (s + t, \gamma_{y,\theta}(t)), \quad  t\in [0, \tau_+(y,\theta)],$$
where $s \in \R$ is a constant. Since the construction of the principal part of the geometric optics solutions is local and takes place near this fixed null geodesic, we introduce  coordinates adapted to it. We first consider polar normal coordinates $(t, r,\theta)$ centered at the point $y$ in $\partial\M_1$, as  introduced above. Since $\M_1$ is simple, the angular variable $\theta$ can be identified  with 
global coordinates $(\theta_2,\cdots,\theta_{n})\in\R^{n-1}$ so that the null geodesic $\tilde{\gamma}$ is written as $(s_0+t, t, 0,\cdots, 0)$ with $t \in  [\tau_-, \tau_+]$. 
To simplify the construction, we introduce null Fermi coordinates  around  $\tilde{\gamma}$ in $\R \times \M$,   defined by
\bel{coco}\left\{
\begin{aligned}
&z_0=&&t+r,\\
&z_1=&& -t+r+s_0,\\  
&z_j=&& \theta_{j} \quad \mbox{ for } j\in\{2,\cdots,n\}.
\end{aligned}
\right.\ee
In these coordinates, the null geodesic $\tilde{\gamma}$ is  simply given by $(z_0,0,\cdots,0)$ for $z_0\in[l_-,l_+]$, with $l_-<l_+$. Here $z_0$ is the affine parameter along the geodesic, while all transverse coordinates vanish. Hence, $\tilde{\gamma}$ lies in the $z_0$-axis. In these coordinates the Lorentzian metric $\tilde{\g}=-dt^2+\g=-dt^2+dr^2+\g_0(r,\theta)$  takes the form
\bel{g(z)} \tilde{\g}=dz_0 dz_1+ \sum_{j,k=2}^n \g_{jk}(z)dz_j dz_k.\ee
We define the tubular neighborhood of  $\tilde{\gamma}$, 
$$V_{\tilde{\gamma},\epsilon}=\{z \in [0,T]\times \M;\; z_0\in [l_-,l_+],\; \mid z'\mid <\epsilon \},$$
where $z'=(z_1,\cdots,z_n)$ and $\abs{z'}=\big(\abs{z_1}^2+\cdots +\abs{z_n}^2\big)^{\frac{1}{2}}$. 
The amplitudes of the geometric optics solutions will be chosen to be supported inside  $V_{\tilde{\gamma},\epsilon}$.
Since $\tilde{\gamma} \subset \mathcal{D} $, one can choose $\epsilon >0$  sufficiently small so that $V_{\tilde{\gamma}, \epsilon} \subset \mathcal{D}$. In particular, the  support of the geometric optics solutions remains away from the hypersurfaces $t=0$ and $t=T$, thereby avoiding boundary effects in the solution of the transport equations.\\
We now rewrite the equations \eqref{eikonal}, \eqref{transport1} and \eqref{transport2} in the Fermi coordinates system \eqref{coco}. Owing to the form of the metric \eqref{g(z)}, the eikonal equation \eqref{eikonal}, is solved by  the phase function \bel{psi}\psi(z)=z_1-s_0,\ee
or, equivalently up to an additive constant, $\psi(z)=z_1$ since the eikonal equation is invariant under the addition of constants. With this choice of phase, the transport equations  \eqref{transport1} and \eqref{transport2} reduce to 
\bel{eqeq}\partial_{z_0} \beta+ \frac{(\partial_{z_0} +\partial_{z_1} )m}{8m}\beta=0,\quad \partial_{z_0} \chi_a+\frac{a}{4}\chi_a=0.\ee
Let us choose $\tilde{a}$ an extension of $a\in C^3([0,T]\times\M)$ to $\R\times\M_1$ such that $\tilde{a}\in W^{3,\infty}(\R;L^\infty(\M_1))\cap W^{1,\infty}(\R;W^{2,\infty}(\M_1))$, supp$(\tilde{a})\subset (-\delta,T+\delta)\times \M_1^{int}$. Then, the functions  
\bel{chi}\beta (z)=m(z)^{-\frac{1}{4}}\mathcal G(z') \quad \mbox{ and  } \quad\chi_a(z)=\exp \big({-\frac{1}{4}} \int_{l_-}^{z_0} \tilde{a}(s,z') ds\big)\ee solve respectively the equations \eqref{eqeq}. 
To ensure that the  geometric optics solutions have support inside the tubular neighborhood $V_{\tilde{\gamma},\epsilon}$, we choose a non-negative function $\Psi\in C_0^\infty(\R^{n-1})$ satisfying  $\Psi(z') = 1$ for $|z'|< \frac{1}{4}$, $\Psi(z') = 0$
for $|z'| > \frac{1}{2}$,  and $\norm{\Psi}_{L^2(\R^{n-1})}= 1$. For $0<\epsilon'<\epsilon$, we define \bel{beta}\beta (z)=m(z)^{-\frac{1}{4}}\Psi(\frac{z'}{\epsilon'}).\ee Then $\supp{\beta}\subset V_{\tilde{\gamma},\epsilon}$, so that the corresponding principal part of geometric optics solutions are localized in an arbitrarily small neighborhood of the null geodesic $\tilde{\gamma}$.

\subsection{Remainder term}
In this section, we consider the construction of the remainder terms
$R_\lambda$ (resp. $R_\lambda^*$) in the geometric optics solutions \eqref{GO} (resp. \eqref{GO*}) with the corresponding decay property \eqref{decay} with respect to $\lambda>0$. For this purpose, let us consider $V_\lambda$ (resp. $V_\lambda^*$) be the principal part of the geometric optics solutions \eqref{GO} (resp. \eqref{GO*}) defined by
$$V_\lambda(t,x)=\beta(t,x)\chi_a(t,x)e^{i\lambda \psi(t,x)},\quad V_\lambda^*(t,x)=\beta(t,x)\chi_{-a}(t,x)e^{-i\lambda \psi(t,x)},\quad (t,x)\in(0,T)\times \M,$$
with $\psi$ given by \eqref{psi}, $\chi_a$ by \eqref{chi} and $\beta$ by \eqref{beta}.
Consider also $F_\lambda$ and $F_\lambda^*$ defined by 
$$F_\lambda(t,x)=-(\partial ^2_t-\Delta_\g +a(t,x) \partial_t +b(t,x) \partial_t^\alpha +c(t,x)))V_\lambda(t,x),\quad (t,x)\in(0,T)\times \M,$$
$$F_\lambda^*(t,x)=-(\partial ^2_t-\Delta_\g -a(t,x) \partial_t - \partial_t^{\alpha*} (b\cdot)-\pa_ta(t,x)+c(t,x)))V_\lambda^*(t,x),\quad (t,x)\in(0,T)\times \M.$$
In view of \eqref{eikonal}, \eqref{transport1}, \eqref{transport2}, for all $(t,x)\in(0,T)\times \M$, we obtain
\bel{F}F_\lambda(t,x)=-e^{i\lambda \psi(t,x)}(\partial ^2_t-\Delta_\g +a(t,x) \partial_t + c(t,x))\beta\chi_{a}(t,x)-b(t,x) \partial_t^\alpha V_\lambda(t,x),\ee
\bel{F*}F_\lambda^*(t,x)=-e^{-i\lambda \psi(t,x)}(\partial ^2_t-\Delta_\g -a(t,x) \partial_t  +(-\pa_ta+c)(t,x))\beta\chi_{-a}(t,x)+ \partial_t^{\alpha*} (bV_\lambda^*)(t,x).\ee 
Using the above property we can define $R_\lambda$ and $R_\lambda^*$ as the solution of the following IBVP
\bel{R}
\left\{
\begin{aligned}
& \partial_t^2 R_\lambda-\Delta_\g R_\lambda +a(t,x)\partial_tR_\lambda+ b(t,x) \partial_t^{\alpha} R_\lambda +c(t,x)R_\lambda  = F_\lambda(t,x), &&\quad (t,x)\in (0,T)\times \M,\\
& R_\lambda(t,x) = 0,                                                                                                      &&\quad (t,x)\in (0,T)\times\pa\M,\\
& R_\lambda(0,x) = 0,\quad \partial_t R_\lambda(0,x) = 0,                                             &&\quad x\in\M,
\end{aligned}
\right.
\ee
\begin{equation}
\label{R*}
\left\{\begin{aligned}
&\partial_t^2 R_\lambda^*-\Delta_\g R_\lambda^* -a\partial_tR_\lambda^*- \partial_t^{\alpha*}(bR_\lambda^*) +(-\pa_t a+c)R_\lambda^*  = F_\lambda^*, &&\quad \textrm{in } (0,T)\times \M,\\
& R_\lambda^*(t,x) = 0,                                                                                                      &&\quad (t,x)\in (0,T)\times\pa\M,\\
& R_\lambda^*(T,x) = 0,\quad \partial_t R_\lambda^*(T,x) = 0,                                             &&\quad x\in\M.
\end{aligned}
\right.
\end{equation}
Using this representation, we can prove the following properties of the remainder terms $R_\lambda$, $R_\lambda^*$.
\begin{prop}\label{p3} Problem \eqref{R} $($resp. \eqref{R*}$)$ admits a unique solution $R_\lambda\in  C([0,T];H^1_0(\M))\cap  C^1([0,T];L^2(\M))$ $($resp. $R_\lambda^*\in  C([0,T];H^1_0(\M))\cap  C^1([0,T];L^2(\M))$$)$ with $\partial_\nu R_\lambda\in L^2((0,T)\times\pa\M)$ $($resp. $\partial_\nu R_\lambda^*\in L^2((0,T)\times\pa\M)$$)$. Moreover, 
   $R_\lambda$ and $R_\lambda^*$ satisfy the decay property \eqref{decay} with respect to $\lambda>1$.
\end{prop}
\begin{proof} We divide the proof into three steps.

\textbf{Step 1.} In this step, we show the regularity properties of $R_\lambda$ and $R_\lambda^*$. Since $F_\lambda, F_\lambda^*\in L^1(0,T;L^2(\M))$, applying Lemma \ref{l1} and \ref{l3}, we deduce $R_\lambda,R_\lambda^*\in  C([0,T];H^1_0(\M))\cap  C^1([0,T];L^2(\M))$. Fixing 
$$G=-(a\partial_tR_\lambda+ b \partial_t^{\alpha} R_\lambda +cR_\lambda),\quad G^*=-(-a\partial_tR_\lambda^*- \partial_t^{\alpha*}(bR_\lambda^*) +(-\pa_ta+c)R_\lambda^*),$$
we observe that $G,G^*\in L^1(0,T;L^2(\M))$ and $R_\lambda$, $R_\lambda^*$ solve respectively the following problems
$$\left\{
\begin{aligned}
& \partial_t^2 R_\lambda-\Delta_\g R_\lambda  = G+F_\lambda(t,x), &&\quad (t,x)\in (0,T)\times \M,\\
& R_\lambda(t,x) = 0,                                                                                                      &&\quad (t,x)\in (0,T)\times\pa\M,\\
& R_\lambda(0,x) = 0,\quad \partial_t R_\lambda(0,x) = 0,                                             &&\quad x\in\M,
\end{aligned}
\right.$$
$$\left\{
\begin{aligned}
& \partial_t^2 R_\lambda^*-\Delta_\g R_\lambda^*  = G^*+F_\lambda^*(t,x), &&\quad (t,x)\in (0,T)\times \M,\\
& R_\lambda^*(t,x) = 0,                                                                                                      &&\quad (t,x)\in (0,T)\times\pa\M,\\
& R_\lambda^*(T,x) = 0,\quad \partial_t R_\lambda^*(T,x) = 0,                                             &&\quad x\in\M.
\end{aligned}
\right.$$
Thus applying \cite[Theorem 2.30]{KKL}, we deduce that $\partial_\nu R_\lambda,\partial_\nu R_\lambda^*\in L^2((0,T)\times\pa\M)$. 

\textbf{Step 2.} In this step, we prove the following estimate
\bel{p3b}\norm{\int_0^tF_\lambda(s,\cdot)ds}_{L^2((0,T)\times\M))}+\norm{\int_{T-t}^TF_\lambda^*(s,\cdot)ds}_{L^2((0,T)\times\M))}\leq C\lambda^{\alpha-1},\ee
with $C>0$ independent of $\lambda>1$. Let us first observe that since $\beta,\chi_{a}\in W^{3,\infty}(0,T;L^\infty(\M))\cap W^{1,\infty}(0,T;W^{2,\infty}(\M))$, we have $(\partial ^2_t-\Delta_\g +a(t,x) \partial_t + c(t,x))\beta\chi_{a}\in W^{1,\infty}(0,T;L^\infty(\M))$. In addition, we have $\partial_t\psi(t,x)=-1$, $(t,x)\in(0,T)\times\M$, and integrating by parts, we get
$$\begin{aligned}&\int_0^te^{i\lambda \psi(s,x)}(\partial ^2_t-\Delta_\g +a \partial_t + c)\beta\chi_{a}(s,x)ds\\
&=\frac{e^{i\lambda \psi(t,x)}(\partial ^2_t-\Delta_\g +a \partial_t + c)\beta\chi_{a}(t,x)}{-i\lambda}+\frac{1}{i\lambda}\int_0^te^{i\lambda \psi(s,x)}\partial_t(\partial ^2_t-\Delta_\g +a \partial_t + c)\beta\chi_{a}(s,x)ds.\end{aligned}$$
This clearly proves that
$$\sup_{t\in[0,T]}\norm{\int_0^te^{i\lambda \psi(s,x)}(\partial ^2_t-\Delta_\g +a \partial_t + c)\beta\chi_{a}(s,\cdot)ds}_{L^2(M)}\leq C\lambda^{-1},$$
with $C>0$ independent of $\lambda>1$. Similarly, we can show that
$$\sup_{t\in[0,T]}\norm{\int_{T-t}^Te^{-i\lambda \psi(s,x)}(\partial ^2_t-\Delta_\g -a(s,x) \partial_t  +(-\pa_ta+c)(s,x))\beta\chi_{-a}(s,\cdot)ds}_{L^2(M)}\leq C\lambda^{-1},$$
and \eqref{p3b} will follow from the following estimate
\bel{p3c} \norm{\int_0^tb\partial_t^\alpha V_\lambda(s,\cdot)ds}_{L^2((0,T)\times\M))} +\norm{\int_{T-t}^T\partial_t^{\alpha*} (bV_\lambda^*)(s,\cdot)ds}_{L^2((0,T)\times\M))}\leq C\lambda^{\alpha-1},  \ee
with $C>0$ independent of $\lambda>1$. 

We start by considering \eqref{p3c} for $V_\lambda$. From now on, for all $\gamma\in(0,1]$, we denote by $I_t^\gamma$ the fractional anti-derivative of order $\gamma$ in $t$ defined by 
\bel{I}I_t^\gamma w(t,x)=\int_0^t\frac{(t-s)^{\gamma-1}}{\Gamma(\gamma)}w(s,x)ds,\quad w\in L^1((0,T)\times\M),\ (t,x)\in(0,T)\times\M.\ee
Since, $V_\lambda(0,\cdot)\equiv0$, one can check that $\partial_t^\alpha V_\lambda=\partial_tI^{1-\alpha}_tV_\lambda$ which implies that
$$\begin{aligned}\abs{\int_0^tb\partial_t^\alpha V_\lambda(s,x)ds}&=\abs{b(t,x)I^{1-\alpha}_tV_\lambda(t,x)-\int_0^t\partial_tbI^{1-\alpha}_tV_\lambda(t,x)ds}\\
&\leq C\norm{I^{1-\alpha}_tV_\lambda}_{L^\infty(0,T;L^\infty(\M))},\quad (t,x)\in(0,T)\times\M,\end{aligned}$$
with $C>0$ independent of $\lambda>1$. This proves that \eqref{p3c} follows from the estimate
\bel{p3d}\norm{I^{1-\alpha}_tV_\lambda}_{L^\infty(0,T;L^\infty(\M))}\leq C\lambda^{\alpha-1},\ee
with $C>0$ independent of $\lambda>1$. Note first that, for all $t\in(0,\min(\lambda^{-1},T))$ and all $x\in\M$, we have
$$I^{1-\alpha}_tV_\lambda(t,x)=\int_0^{t}\frac{(t-s)^{-\alpha}}{\Gamma(1-\alpha)}V_\la(s,x)ds=\int_0^{t}\frac{s^{-\alpha}}{\Gamma(1-\alpha)}V_\la(t-s,x)ds$$
which implies that
\bel{p3e}\norm{I^{1-\alpha}_tV_\lambda}_{L^\infty(0,\min(\lambda^{-1},T);L^\infty(\M))}\leq C\int_0^{\lambda^{-1}}s^{-\alpha}ds=C\lambda^{\alpha-1},\ee
with $C>0$ independent of $\lambda>1$. Similarly, for $\lambda>T^{-1}$ and all $t\in(\lambda^{-1},T)$ and all $x\in\M$, we have
$$\begin{aligned}\abs{I^{1-\alpha}_tV_\lambda(t,x)}&\leq \abs{\int_0^{\lambda^{-1}}\frac{s^{-\alpha}}{\Gamma(1-\alpha)}V_\la(t-s,x)ds}+\abs{\int_{\lambda^{-1}}^t\frac{s^{-\alpha}}{\Gamma(1-\alpha)}V_\la(t-s,x)ds}\\
&\leq C\lambda^{\alpha-1}+\abs{e^{i\lambda \psi(t,x)}\int_{\lambda^{-1}}^te^{-i\lambda s}\frac{s^{-\alpha}}{\Gamma(1-\alpha)}\beta\chi_a(t-s,x)ds}\\
&\leq C\lambda^{\alpha-1}+\abs{\frac{e^{-i}}{i\lambda}\cdot\frac{\lambda^{\alpha}}{\Gamma(1-\alpha)}\cdot\beta\chi_a(t-\lambda^{-1},x) +\frac{1}{i\lambda}\int_{\lambda^{-1}}^te^{-i\lambda s}\partial_s\left(\frac{s^{-\alpha}}{\Gamma(1-\alpha)}\beta\chi_a(t-s,x)\right)ds}\\
&\leq C\left(\lambda^{\alpha-1}+\lambda^{-1}\left(\int_{\lambda^{-1}}^t(s^{-1-\alpha}+s^{-\alpha})ds\right)\right)\leq C\lambda^{\alpha-1},
\end{aligned}$$
with $C>0$ independent of $\lambda>1$. Combining this with \eqref{p3e}, we obtain
$$\norm{I^{1-\alpha}_tV_\lambda}_{L^\infty(0,T;L^\infty(\M))}\leq C\lambda^{\alpha-1},$$
which implies that
\bel{p3f}\norm{\int_0^tb\partial_t^\alpha V_\lambda(s,\cdot)ds}_{L^2((0,T)\times\M))}\leq C\lambda^{\alpha-1},\ee
with $C>0$ independent of $\lambda>1$. Similarly,   applying Lemma \ref{l2}, we find
$$\partial_t^{\alpha*} (bV_\lambda^*)(t,x)=\partial_t^{\alpha}H_\lambda(T-t,x),\quad (t,x)\in(0,T)\times\M,$$
with
$$H_\lambda(t,x)=bV_\lambda^*(T-t,x)=b(T-t,x)\beta(T-t,x)\chi_{-a}(T-t,x)e^{-i\lambda \psi(T-t,x)},\quad (t,x)\in(0,T)\times\M.$$
Thus, recalling that $V_\lambda^*(T,\cdot)\equiv0$ we deduce that $H_\lambda(0,\cdot)\equiv0$ and, repeating the above argumentation, we deduce that
$$\begin{aligned}\norm{\int_{T-t}^T\partial_t^{\alpha*} (bV_\lambda^*)(s,\cdot)ds}_{L^2((0,T)\times\M))}&=\norm{\int_{T-t}^T\partial_t^{\alpha}H_\lambda(T-s,x)ds}_{L^2((0,T)\times\M))}\\ 
&=\norm{\int_0^{t}\partial_t^{\alpha}H_\lambda(s,x)ds}_{L^2((0,T)\times\M))}\leq C\lambda^{\alpha-1},\end{aligned}$$
with $C>0$ independent of $\lambda>1$. Combining this estimate with \eqref{p3f}, we obtain \eqref{p3c} and by the same way \eqref{p3b}.

\textbf{Step 3.} In this step, we complete the proof of the proposition. Applying Lemma \ref{l1}, \ref{l3} as well as formula \eqref{F}-\eqref{F*}, we obtain
\bel{p3g}\begin{aligned}
\norm{R_\lambda}_{C^1([0,T];L^2(\M))}+\norm{R_\lambda^*}_{C^1([0,T];L^2(\M))}
&\leq C(\norm{F_\lambda}_{L^1(0,T;L^2(\M))}+\norm{F_\lambda^*}_{L^1(0,T;L^2(\M))})\\
&\leq C\left(1+\norm{\partial_t^\alpha V_\lambda}_{L^1(0,T;L^2(\M))}+\norm{\partial_t^{\alpha*} (bV_\lambda^*)}_{L^1(0,T;L^2(\M))}\right),\end{aligned}\ee
with $C>0$ independent of $\lambda>1$. Recalling that $\partial_t^\alpha V_\lambda=I^{1-\alpha}_t\partial_tV_\lambda$, $$\partial_tV_\lambda(t,x)=-i\lambda V_\lambda(t,x)+e^{i\lambda \psi(t,x)}\partial_t(\beta\chi_a)(t,x),\quad (t,x)\in(0,T)\times\M$$
and repeating the argumentation of Step 2., we obtain
$$\norm{\partial_t^\alpha V_\lambda}_{L^1(0,T;L^2(\M))}\leq C\norm{I^{1-\alpha}_t\partial_tV_\lambda}_{L^\infty(0,T;L^\infty(\M))}\leq C\lambda^\alpha.$$
Similarly, we obtain
$$\norm{\partial_t^{\alpha*} (bV_\lambda^*)}_{L^1(0,T;L^2(\M))}\leq C\lambda^{\alpha}$$
and, using \eqref{p3g}, we find
\bel{p3h}
\norm{R_\lambda}_{C^1([0,T];L^2(\M))}+\norm{R_\lambda^*}_{C^1([0,T];L^2(\M))}\leq 
C\lambda^\alpha,\ee
with $C>0$ independent of $\lambda>1$. Finally, in view of Proposition \ref{p2} and estimate \eqref{p3b}, we have
$$\begin{aligned}&\norm{R_\lambda}_{C([0,T];L^2(\M))}+\norm{R_\lambda^*}_{C([0,T];L^2(\M))}\\
&\leq C\left(\norm{\int_0^tF_\lambda(s,\cdot)ds}_{L^2((0,T)\times\M))}+\norm{\int_{T-t}^TF_\lambda^*(s,\cdot)ds}_{L^2((0,T)\times\M))}\right)\leq C\lambda^{\alpha-1}.\end{aligned}$$
Combining this estimate with \eqref{p3h}, we obtain the estimates \eqref{decay}.\end{proof}
\subsection{Geometric optics solutions for our main results}\label{GON}

In this section, we consider the geometric optics solutions required for our main results. More precisely, we choose $a_j,c_j,b \in C^3([0,T]\times \M)$, $j=1,2$, and we assume that condition \eqref{t1b} is fulfilled. We assume also that either \eqref{t2a} is fulfilled or $a_j,c_j$ are time independent. We will consider geometric optics $u_{2,\lambda}$ of the form \eqref{GO} and $u_{1,\lambda}^*$ of the form \eqref{GO*} solving the following problem
\begin{equation}
\label{eq-GO}
\left\{
\begin{aligned}
& \partial_t^2 u_{2,\lambda}-\Delta_\g u_{2,\lambda} +a_2(t,x)\partial_tu_{2,\lambda}+ b(t,x) \partial_t^{\alpha} u_{2,\lambda} +c_2(t,x)u_{2,\lambda}  = 0, &&\quad (t,x)\in (0,T)\times \M,\\
& u_{2,\lambda}(0,x) = 0,\quad \partial_t u_{2,\lambda}(0,x) = 0,                                             &&\quad x\in\M,
\end{aligned}
\right.
\end{equation}
\begin{equation}
\label{eq-GO*}
\left\{
\begin{aligned}
& \partial_t^2 u_{1,\lambda}^*-\Delta_\g u_{1,\lambda}^* -a_1(t,x)\partial_tu_{1,\lambda}^*-  \partial_t^{\alpha*} (bu_{1,\lambda}^*) +(-\pa_ta_1(t,x)+c_1(t,x))u_{1,\lambda}^*  = 0, &&\quad (t,x)\in (0,T)\times \M,\\
& u_{1,\lambda}^*(T,x) = 0,\quad \partial_t u_{1,\lambda}^*(T,x) = 0,                                             &&\quad x\in\M.
\end{aligned}
\right.
\end{equation}
For this purpose, applying \cite[Section 3, Theorem 5]{St} in the case of a smooth Riemanian manifold, we can find $\tilde{a}_1 \in W^{3,\infty}(\R;L^\infty(\M_1))\cap W^{1,\infty}(\R;W^{2,\infty}(\M_1))$ such that supp$(\tilde{a_1})\subset (-\delta,T+\delta)\times \M_1^{int}$ and $\tilde{a}_1=a_1$ on $[0,T]\times\M$. 
Next, we put
$$ \tilde{a}_2(t,x) :=
\left\{
\begin{array}{ll} 
a_2(t,x), & \mbox{if}\ (t,x)\in[0,T]\times M, \\ \tilde{a}_1(t,x), & \mbox{if}\ (t,x) \in \R\times\M_1\setminus ([0,T]\times M) . 
\end{array}\right.$$
Then, according to \eqref{t1a} and \eqref{t1b}, we have
$\tilde{a}_2 \in W^{3,\infty}(\R;L^\infty(\M_1))\cap W^{1,\infty}(\R;W^{2,\infty}(\M_1))$. In addition, choosing $a=a_1-a_2$ extended by zero to $\R\times\M_1$ we have $a=\tilde{a}_1-\tilde{a}_2$.

Using these properties as well as the results of the preceding section,  to every maximal null geodesic $\tilde{\gamma} \subset  \mathcal D $ we can associate geometric optics solutions  of \eqref{eq-GO}-\eqref{eq-GO*} taking the form
\bel{GO1}u_{2,\lambda}(t,x)=\beta(t,x)\chi_{a_2}(t,x)e^{i\lambda \psi(t,x)}+R_\lambda(t,x),\quad 
u_{1,\lambda}^*(t,x)=\beta(t,x)\chi_{-a_1}(t,x)e^{-i\lambda \psi(t,x)}+R_\lambda^*(t,x),\ee
where $\psi$ is given by \eqref{psi}, $\beta$ is given by \eqref{beta}, $\chi_{\pm a_j}$, $j=1,2$, is given by \eqref{chi} with $\tilde{a}=\tilde{a}_j$ and the remainder terms $R_\lambda,R_\lambda^*$ solving respectively \eqref{R}, with $a=a_2$ and $c=c_2$, and \eqref{R*}, with $a=a_1$ and $c=c_1$, and satisfying the properties stated in Proposition \ref{p3}.

\section{Proof the main results}\label{s5}
This section will be devoted to the proof of the main results stated in Theorem \ref{t1}, \ref{t2} and \ref{t3}. For this purpose, we start by considering a reduction of the data which relies on the memory effect of equation \eqref{eq1} arising from the presence of the fractional attenuation.

\subsection{Transfer of data}\label{s5.1}

The main goal of this subsection is to show the fundamental result of transfer of the data from the partial hyperbolic Dirichlet-to-Neumann map $\Lambda_{a,c,\epsilon}$, with $\epsilon\in(0,T)$ arbitrary small, and the  hyperbolic Dirichlet-to-Neumann $\mathcal N_{a,c}$ defined on the full lateral boundary. This result can be stated as follows.

\begin{lem}\label{l4} For $j=1,2$, let $a_j,b,c_j \in C^3([0,T]\times \M)$ and assume that there exists $\epsilon_0\in(0,T)$ such that condition \eqref{t1c} and the following  
\bel{l4a} \partial_\nu^ka_1(t,x)=\partial_\nu^ka_2(t,x),\quad \partial_\nu^kc_1(t,x)=\partial_\nu^kc_2(t,x),\quad   k=0,1,\ (t,x)\in(T-\epsilon_0,T)\times\pa\M,\ee
are fulfilled. Then, for any $\epsilon\in(0,\epsilon_0)$, the following implication 
\bel{l4b}\Lambda_{a_1,c_1,\epsilon}=\Lambda_{a_2,c_2,\epsilon}\Rightarrow\mathcal N_{a_1,c_1}=\mathcal N_{a_2,c_2}\ee
holds true.
    
\end{lem}
\begin{proof} Fix $\epsilon\in(0,\epsilon_0)$ and assume that 
\bel{l4c} \Lambda_{a_1,c_1,\epsilon}=\Lambda_{a_2,c_2,\epsilon}.\ee
Now choose $f\in\mathcal J\cap H^1_0((0,T)\times\pa\M)$ and, applying Theorem \ref{t4},  consider $u_j\in C([0,T];H^2(\M))\cap C^2([0,T];H^4(\M))$, $j=1,2$, solving \eqref{eq1} with $a=a_j$ and $c=c_j$. Then, in view of \eqref{l4c}, we have
\bel{l4d} \partial_\nu u_1(t,x)=\partial_\nu u_2(t,x),\quad \partial_\nu\Delta_\g u_1(t,x)=\partial_\nu\Delta_\g u_2(t,x),\quad (t,x)\in[T-\epsilon,T]\times\pa\M.\ee
Fix $u=u_1-u_2$ and observe that $u\in C([0,T];H^4(\M))\cap C^2([0,T];H^2(\M))$
solves the problem
$$
\left\{
\begin{aligned}
& \partial_t^2 u-\Delta_\g u +a_1(t,x)\partial_tu+ b(t,x) \partial_t^{\alpha} u +c_1(t,x)u  =(a_2-a_1) \partial_tu_2+(c_2-c_1) u_2, &&\quad (t,x)\in (0,T)\times \M,\\
& u(t,x) = 0,                                                                                                      &&\quad (t,x)\in (0,T)\times\pa\M,\\
& u(0,x) = 0,\quad \partial_t u(0,x) = 0,                                             &&\quad x\in\M.
\end{aligned}
\right.$$
It follows that 
$$\begin{aligned}&(\partial_t^2\pa_\nu u-\pa_\nu\Delta_\g u+\pa_\nu a_1\partial_tu+a_1\partial_t\pa_\nu u+\pa_\nu b\partial_t^{\alpha} u+b\partial_t^{\alpha} \pa_\nu u+\pa_\nu c_1u+c_1\pa_\nu u)(t,x)\\
&=\pa_\nu (\partial_t^2 u-\Delta_\g u +a_1\partial_tu+ b \partial_t^{\alpha} u +c_1u )(t,x)\\
&=(\pa_\nu(a_2-a_1)\pa_tf+(a_2-a_1)\pa_\nu \pa_tu_2+\pa_\nu(c_2-c_1)f+(c_2-c_1)\pa_\nu u_2)(t,x),\quad (t,x)\in(T-\epsilon,T)\times\pa\M.\end{aligned}$$
Applying \eqref{l4a}, \eqref{l4d} and using the fact that $u(t,x)=0$, $(t,x)\in(0,T)\times\pa\M$, we obtain
$$b(t,x)\partial_t^{\alpha} \pa_\nu u(t,x)=0,\quad (t,x)\in(T-\epsilon,T)\times\pa\M.$$
Then, condition \eqref{t1c} implies
$$\pa_\nu u(t,x)=\partial_t^{\alpha} \pa_\nu u(t,x)=0,\quad (t,x)\in(T-\epsilon,T)\times\pa\M.$$
Applying Theorem \ref{memory} in the Appendix (see also  \cite[Theorem 1]{JK}) and \eqref{RC}, we obtain
$$\pa_\nu u_1(t,x)-\pa_\nu u_2(t,x)= \pa_\nu u(t,x)=0,\quad (t,x)\in(0,T)\times\pa\M.$$
This clearly implies that $\mathcal N_{a_1,c_1}f=\mathcal N_{a_2,c_2}f$, $f\in\mathcal J\cap H^1_0((0,T)\times\pa\M)$. Finally, using the density of   $\mathcal J\cap H^1_0((0,T)\times\pa\M)$ in $H^1_0((0,T)\times\pa\M)$, we deduce that $\mathcal N_{a_1,c_1}=\mathcal N_{a_2,c_2}$. This proves that the implication \eqref{l4b} holds true and it completes the proof of the lemma.\end{proof}
Similarly to Lemma \ref{l4}, we can prove the following.
\begin{lem}\label{l5} Let the condition of Lemma \ref{l4} be fulfilled and assume that \eqref{t2b} holds true.
Then, we have
\bel{l5a}\mathcal N_{a_1,c_1}f=\mathcal N_{a_2,c_2}f,\quad f\in H^1_0((0,T-\epsilon_1)\times\pa\M).\ee
\end{lem}

The remaining parts of this section will be devoted to the proof of Theorem \ref{t1}, \ref{t2}, \ref{t3} by applying the transfer of data stated in Lemma \ref{l4} and \ref{l5}.

\subsection{Proof of Theorem \ref{t1}}
In view of Lemma \ref{l4},  we are left with the task of proving that the condition
\bel{N}\mathcal N_{a_1,c_1}=\mathcal N_{a_2,c_2}\ee
implies that $a_1=a_2$ and $c_1=c_2$. We divide the proof into three steps. Namely, we start by reformulating the problem in terms of an orthogonality identity. Then, we prove the recovery of the first order coefficient $a$. Finally, we consider the determination of the potential $c$.
\subsubsection{An orthogonality identity}
Fix $f\in \mathcal J\cap H^1_0((0,T)\times\pa\M)$ and $h\in H^1_0((0,T)\times\pa\M)$ and consider $u_j\in C^2([0,T];H^2(\M))$, $j=1,2$ solving \eqref{eq1} with $a=a_j$, $c=c_j$. Thus, $u=u_1-u_2$ satisfies the following conditions
$$
\left\{
\begin{aligned}
& \partial_t^2 u-\Delta_\g u +a_1(t,x)\partial_tu+ b(t,x) \partial_t^{\alpha} u +c_1(t,x)u  =(a_2-a_1) \partial_tu_2+(c_2-c_1) u_2, &&\quad (t,x)\in (0,T)\times \M,\\
& \partial_\nu u(t,x)=u(t,x) = 0,                                                                                                      &&\quad (t,x)\in (0,T)\times\pa\M,\\
& u(0,x) = 0,\quad \partial_t u(0,x) = 0,                                             &&\quad x\in\M.
\end{aligned}
\right.$$
Thus, multiplying this equation by the solution $u_1^*\in C([0,T];H^1_0(\M))\cap C^1([0,T];L^2(\M))$ of \eqref{eq1*} with $a=a_1$, $c=c_1$, that satisfies $\partial_\nu u_1^*\in L^2((0,T)\times\pa\M)$, integrating by parts and applying Lemma \ref{l2}, we obtain
$$\begin{aligned}&\int_0^T\int_\M((a_2-a_1) \partial_tu_2u_1^*+(c_2-c_1) u_2u_1^*)dV_\g(x)dt\\
&=\int_0^T\int_\M(\partial_t^2 u-\Delta_\g u +a_1(t,x)\partial_tu+ b(t,x) \partial_t^{\alpha} u +c_1(t,x)u)u_1^*dV_\g(x)dt\\
&=\int_0^T\int_\M u (\partial_t^2 u_1^*-\Delta_\g u_1^* -a_1\partial_tu_1^*-  \partial_t^{\alpha*} (bu_1^*) +(-\pa_ta_1+c_1)u_1^*)dV_\g(x)dt=0.\end{aligned}$$
This leads to the following orthogonality identity
\bel{ortho}\int_0^T\int_\M((a_2-a_1) \partial_tu_2u_1^*+(c_2-c_1) u_2u_1^*)dV_\g(x)dt=0.\ee
Moreover, by density of $\mathcal J\cap H^1_0((0,T)\times\pa\M)$ in $H^1_0((0,T)\times\pa\M)$, we can deduce that \eqref{ortho} holds true for any solution $u_2$ of \eqref{eq1} with $a=a_2$, $c=c_2$, $f\in H^1_0((0,T)\times\M)$ and  any solution $u_1^*$ of \eqref{eq1*} with $a=a_1$, $c=c_1$, $h\in H^1_0((0,T)\times\M)$.

\subsubsection{Recovery of the coefficient $a$}
We now prove that the  coefficient $a=a_1-a_2$ extended by zero to $\R\times\M_1$ is null. For this purpose, starting from the  orthogonality identity \eqref{ortho}, we show that the light ray transform of $a$ vanishes along every maximal null geodesic. This is achieved by using the geometric optic solutions of the preceding section. The conclusion $a\equiv0$ then follows from the injectivity of the light ray transform, which we now recall.
\begin{definition}
Let $(\R \times \M_1, -dt^2+\g_1)$ be the  Lorentzian manifold associated with $(\M_1,\g_1)$. For  $h \in C_0([0,T] \times \M_1)$, the light ray transform of $h$  along a maximal null geodesic  $\widetilde{\gamma}$ of $\R \times \M_1$ is defined by
\begin{equation}\label{light ray transform def}
\mathcal{L}_{\widetilde{\gamma}} h=\int_{\mathbb{R}} h(\widetilde{\gamma}(s)) ds.
\end{equation}
\end{definition}
 It is well known from  \cite[Proposition 1.4]{FIKO} and \cite{SU}, that since  $(\M_1,\g)$ is simple, the light ray transform $\mathcal L$ is injective. More precisely,  the following implication
\begin{equation}\label{injectivity light ray}
\mathcal{L}_{\widetilde{\gamma}} h=0 \;\mbox{  for all maximal null geodesic  }\; \widetilde{\gamma}\; \mbox{ in } \;\R \times \M \Longrightarrow h \equiv 0
\end{equation}
holds true.  Substituting the geometric optics solutions  \eqref{GO1} in the orthogonality identity \eqref{ortho},  we obtain
$$\begin{aligned}-i \lambda \int_0^T\int_\M a \beta^2 \chi_a dV_\g(x)dt=& -\int_0^T\int_\M \big(a \beta  \chi_{-a_1} \partial_t (\beta \chi_{a_2}) +a \beta \chi_{-a_1} e^{-i \lambda  \psi(t,x)} \partial_t R_{\lambda}
\big) dV_\g(x)dt\\
&-\int_0^T\int_\M \big(a R_\lambda^* \partial_t(\beta \chi_{a_2})e^{i \lambda \psi(t,x)} -i \lambda a \beta \chi_{a_2} e^{i \lambda \psi(t,x)}R_\lambda^*\big) dV_\g(x)dt\\
&-\int_0^T\int_\M \big(a R_\lambda^* \partial_t R_\lambda+c \beta^2 \chi_a+c \beta \chi_{a_2} R_\lambda^* e^{i \lambda \psi(t,x)} \big)  dV_\g(x)dt\\
&-\int_0^T\int_\M \big(c \beta \chi_{-a_1} R_\lambda e^{-i \lambda \psi(t,x)}+ c R_\lambda R_\lambda^*\big)dV_\g(x)dt\\
&:=I_\lambda,
\end{aligned}$$
with $c=c_1-c_2$. Using the regularity of  $a_j,c_j,\beta, \chi_{a_j},\; j=1,2$, together with the decay estimate (\ref{decay}) satisfied by $R_\lambda$ and $R_\lambda^*$,  there exists a constant $C>0$, independent of $\lambda>1$, such that 
$\vert{ I_\lambda }\vert \leq C \lambda^{\alpha }$.
Since $0<\alpha<1$, dividing by $\lambda$ and letting $\lambda \rightarrow \infty$ yields 
\bel{I=0}\int_0^T\int_\M a(t,x) \beta^2(t,x) \chi_a(t,x)  dV_\g(x)dt=0.\ee
We next prove that $$\mathcal{L}_{\tilde{\gamma}}a=0,$$
for every maximal null geodesic $\tilde{\gamma}$. We will use notations  provided by Section \ref{section GOS}.
Let $\tilde{\gamma}\subset{D}$ be an arbitrary maximal null geodesic, and let $\gamma$ denote its projection  onto $\M$.
 Since $\M \subset \M_1^{int}$, the geodesic $\gamma$ extends uniquely to a maximal unit-speed geodesic in  the simple manifold $(\M_1,\g)$. 
Let $y\in \partial  \M_1 \cup \gamma$  be its entrance point into $\M_1$. Then there exists a unique inward-pointing unit vector $\theta \in \partial _+ S\M_1$  such that $\gamma= \gamma_{y,\theta}$, defined on the maximal interval $[0,\tau_+((y,\theta)]$. The corresponding null geodesic in $\R \times \M_1$ is
$$\tilde{\gamma}(t)=(s+t,\gamma_{y,\theta}(t)),\quad 0\leq t \leq \tau_+(y,\theta),$$ 
where $s\in \R$ is fixed. Following the construction of Section \ref{section GOS} and \ref{GON} in the coordinates \eqref{coco}, recalling that $m^{-\frac{1}{2}} dV_\g dt=  dz_0~dz'$ and that $a$ is supported in $\mathcal E_T$, we can rewrite the  identity \eqref{I=0} as follows
$$\begin{aligned}0=&\int_{\{ \mid z' \mid < \epsilon\}}~\int_{l_-}^{l_+}\left(\Psi(z'/\epsilon')\right)^2a(z_0,z')\exp \big({-\frac{1}{4}} \int_{l_-}^{z_0} a(s,z') ds\big)~dz_0~dz'\\
=&\int_{\{ \mid z' \mid < \epsilon\}}\left(\Psi(z'/\epsilon')\right)^2\big[\exp \big({-\frac{1}{4}} \int_{l_-}^{l_+} a(s,z') ds\big)-1\big]~dz'.\end{aligned}$$
Sending $\epsilon' \rightarrow 0$ gives 
$$\exp \big({-\frac{1}{4}} \int_{l_-}^{l_+} a(z_0,0) dz_0\big)=1,$$
and therefore
$$\mathcal{L}_{\tilde{\gamma}}a=\int_{l_-}^{l_+} a(z_0,0) dz_0 \in 8i \pi \Z.$$
Since $a$ is real-valued, it follows 
$\mathcal{L}_{\tilde{\gamma}}a=0$.
In view of condition \eqref{t1a}, we have $\supp{(a)} \subset \mathcal{E}_T$ and we infer that 
$\mathcal{L}_{\tilde{\gamma}}a=0$ for every maximal null geodesic  $\tilde{\gamma}$ in $\R\times\M$. Then, the injectivity of the light ray transform  implies that $a\equiv 0$, 
 which completes the proof of the recovery of the damping coefficient $a$.
 
 \subsubsection{Recovery of the coefficient $c$}
  Similarly to the above argumentation, let us first fix $\tilde{\gamma}\subset{D}$ be an arbitrary maximal null geodesic. Since we have already proved that $a_1=a_2$, the orthogonality identity \eqref{ortho} simplifies to $$\int_0^T~\int_\M~c(t,x)u_2(t,x) u_1^*(t,x)\;  dV_g(x)\; dt=0,$$
where $c=c_2-c_1$, extended by zero to $\R\times \M_1$. We choose the same geometric optics solutions as in Section \ref{section GOS} and \ref{GON} and, applying again the decay estimates (\ref{decay}) and sending $\lambda \rightarrow +\infty$, we find  $$\int_0^T~\int_\M~c(t,x)\beta^2(t,x)\;  dV_g(x)\; dt=0.$$
Stating this identity in the null Fermi coordinates introduced above, we obtain 
 $$\int_{\{ \mid z' \mid < \epsilon\}}~\int_{l_-}^{l_+}(\Psi(z'/\epsilon'))^2~c(z_0,z')~dz_0~dz'=0.$$
Then, sending $\epsilon' \rightarrow 0$ gives
 $$\int_{l_-}^{l_+}c(z_0,0)~dz_0= \mathcal{L}_{\tilde{\gamma}}~c =0.$$
 Combining this identity with condition \eqref{t1a}, we deduce that that 
$\mathcal{L}_{\tilde{\gamma}}c=0$ for every maximal null geodesic  $\tilde{\gamma}$ in $\R\times\M$. Then, the injectivity of the light ray transform implies   $c\equiv0$. This completes the proof of Theorem \ref{t1}.
\subsection{Proof of Theorem \ref{t2} and \ref{t3}}
Assume that the hypotheses of Theorems \ref{t2} and \ref{t3} are satisfied, and that \eqref{t2b}
holds  for some  arbitrary small $\epsilon\in(0,\epsilon_1)$. We also  assume that the conditions \eqref{t1b}-\eqref{t1c} hold, together with  \eqref{t2a} in the time-dependent case and  \eqref{t3a} in the time-independent case.

Arguing exactly as in Section \ref{s5.1} and applying Lemma ~\ref{l5}, we obtain  $\mathcal N_{a_1,c_1}f=\mathcal N_{a_2,c_2}f$  for all $f\in J_{\epsilon_1}\cap H^1_0((0,T-\epsilon_1)\times\pa\M)$. Since $J_{\epsilon_1}\cap H^1_0((0,T-\epsilon_1)\times\pa\M)$ is dense in $H^1_0((0,T-\epsilon_1)\times\pa\M)$, it follows that $\mathcal N_{a_1,c_1}f=\mathcal N_{a_2,c_2}f$,  for all $f\in H^1_0((0,T-\epsilon_1)\times\M)$. Repeating the arguments of  Theorem~\ref{t1}, we obtain  the orthogonality identity
\bel{ortho-bis}\int_0^{T-\epsilon_1}\int_\M((a_2-a_1) \partial_tu_2u_1^*+(c_2-c_1) u_2u_1^*)dV_\g(x)dt=0,\ee
for every solution $u_2$ of \eqref{eq1} with $a=a_2$, $c=c_2$, $f\in H^1_0((0,T-\epsilon_1)\times\M)$ and  every solution $u_1^*$ of the adjoint equation \eqref{eq1*} associated with $a=a_1$, $c=c_1$.
We now use \eqref{ortho-bis} to complete the proofs of  Theorems \ref{t2} and \ref{t3}.

\textbf{Proof of Theorem \ref{t2}.}
The proof follows the same strategy as that of Theorem \ref{t1}. Since $$\supp{a}~ \cup ~\supp{c} \subset \mathcal{E}_{T-\epsilon_0},$$
every maximal null geodesic intersecting the support of the unknown coefficients admits a tubular neighborhood contained in 

$$\mathcal D_{T-\epsilon_0}:=\{(t,x)\in (0,T-\epsilon_0)\times\M :\  \textrm{dist}{(x,\pa \M)}<t<T-\epsilon_0-\textrm{dist}{(x,\pa \M)}\}.$$ Consequently, the geometric optics solutions constructed in Section~\ref{section GOS} may
be chosen with support entirely contained in this neighborhood and therefore
away from the hypersurfaces $t=0$ and $t=T-\epsilon_0$. Moreover, since $\epsilon <\frac{\min(\epsilon_0,\epsilon_1)}{2}$, the boundary traces of the forward geometric optics solutions are supported in $[0,T-\epsilon_1]\times \pa \M$, while the adjoint solutions vanish on the portion of the boundary where no measurements are available. Hence, the construction of the geometric optics solutions and the limiting
argument of the proof of Theorem~\ref{t1} remain valid without any
modification.

Substituting these geometric optics solutions into \eqref{ortho-bis}  and letting $\lambda \to +\infty$  yields exactly the same integral identities as in the proof of Theorem \ref{t1}. Passing to null Fermi coordinates  then shows that  the light ray transforms of both $a$ and $c$  vanish along every maximal null geodesic. The injectivity of the light ray transform therefore implies $a=c=0$, which completes the proof.

\textbf{Proof of Theorem \ref{t3}.}
Assume now that the coefficients $a_j$  and $c_j$  are independent of the time variable. They can therefore be  regarded as functions on $[0,T]\times \M$ that are constant with respect to $t$. The assumptions \eqref{t1b} and \eqref{t3a}  ensure the  conditions required to derive the orthogonality identity \eqref{ortho-bis}. 

The geometric optics solutions constructed in the proof of
Theorem~\ref{t2} are still applicable, since their amplitudes are localized
near maximal null geodesics contained in the accessible region and their
boundary traces belong to $J_{\epsilon_1}$.
Repeating the previous argument, we conclude that the light ray transforms
of $a$ and $c$ vanish along every maximal null geodesic.

Since $a=a_1-a_2$ and $c=c_1-c_2$ are independent of $t$, the light ray transform reduces,
up to an inessential multiplicative constant arising from the affine
parametrization of the null geodesic, to the geodesic ray transform on  $(\M_1,\g_1)$ of $a$, extended by zero to $\M_1$, is defined by $\mathcal{I}_{\gamma}(a)=\int_0^{\tau_+(y,\theta)} a(\gamma(s))~ds$. Therefore, $$\mathcal L_{\tilde{\gamma}}a=0 \quad\Longleftrightarrow\quad
\mathcal I_{\gamma}a=0.$$ Since the geodesic ray transform is injective on simple manifolds, we obtain $a\equiv0$. The same argument applies to $c$, yielding $c\equiv0$. This completes the proof.

\appendix
\section{Appendix}
Let us first recall the definition of the Riemann-Liouville time fractional derivative $D^\alpha_t$ of order $\alpha$
defined on $L^1(0,T; L^2(X))$, $X=\M$ or $X=\pa\M$, by
$$D^\alpha_tw(t,x)=\partial_tI^{1-\alpha}_tw(t,x),\quad w\in L^1(0,T; L^2(X)),\ (t,x)\in(0,T)\times X.$$
Here the fractional time integral operator $I^{1-\alpha}_t$ is defined by \eqref{I}, with $\gamma=1-\alpha$, and, observing that, for all $w\in L^1(0,T; L^2(X))$,  $I^{1-\alpha}_tw\in L^1(0,T; L^2(X))$, we deduce that $D^\alpha_tw$ can be seen as an element of $D'(0,T;L^2(X))$. Note also that we have
\bel{RC}D^\alpha_t[w-w(0,\cdot)](t,x)=\partial_t^\alpha w(t,x),\quad w\in W^{1,1}(0,T; L^2(X)),\ (t,x)\in(0,T)\times X.\ee
We consider the following result, which highlights the memory effect associated with the time-fractional derivative. To the best of our knowledge, it is established in the most general setting and under the weakest regularity assumptions available to date.

\begin{thm}\label{memory}
Let $w\in L^1(0,T;L^2(X))$, $h\in L^2(X)$, and fix $\epsilon\in(0,T)$. Then the condition
\bel{tma}
w(t,x)=D_t^\alpha (w-h(x))(t,x)=0,\ (t,x)\in (T-\epsilon,T)\times X\ee
implies that 
\bel{tmab}w(t,x)=h(x)=0,\ (t,x)\in (0,T)\times X.\ee

\end{thm}

\begin{proof}
A similar result was proved in \cite[Step 1, Theorem 1.1]{HJKSZ} under stronger regularity assumptions. Here, we extend the analysis of \cite[Step 1, Theorem 1.1]{HJKSZ} by substantially relaxing the regularity assumptions and providing an alternative proof. Assume that \eqref{tma} holds true. We extend $w$ by zero to a map defined on $(0,+\infty)\times X$ still denoted by $w$, we fix $\psi\in L^2(X)$  and we consider $$w_\psi(t)=\left\langle w(t,\cdot),\psi\right\rangle_{L^2(X)},\quad h_\psi=\left\langle h,\psi\right\rangle_{L^2(X)},\quad t\in(0,+\infty).$$
Then, we find
$$I^{1-\alpha}_tw_\psi(t)=\int_0^{T-\epsilon}\frac{(t-s)^{-\alpha}}{\Gamma(1-\alpha)}w_\psi(s)ds,\quad I^{1-\alpha}_th_\psi(t)=h_\psi\frac{t^{1-\alpha}}{\Gamma(2-\alpha)},\quad t\in (T-\epsilon,+\infty)$$
and one can check that $t\mapsto I^{1-\alpha}_t(w_\psi-h_\psi)(t)$ is analytic on $(T-\epsilon,+\infty)$. On the other hand, we have
$$\partial_tI^{1-\alpha}_t(w_\psi-h_\psi)(t)=D_t^\alpha (w_\psi-h_\psi)(t)=\left\langle D_t^\alpha (w-h)(t,\cdot),\psi\right\rangle_{L^2(X)}=0,\quad t\in (T-\epsilon,T)$$
and by isolated zero theorem we deduce that
$\partial_tI^{1-\alpha}_t(w_\psi-h_\psi)=0$ on $(T-\epsilon,+\infty)$. It follows that 
$$p\int_T^{+\infty}\underbrace{I^{1-\alpha}_t(w_\psi-h_\psi)(t)}_{=I^{1-\alpha}_t(w_\psi-h_\psi)(T)}e^{-pt}dt=I^{1-\alpha}_t(w_\psi-h_\psi)(T)e^{-pT},\quad p>0.$$
In addition, for all $p\in(0,+\infty)$, we find
$$\begin{aligned}p\int_0^{+\infty}e^{-pt}I^{1-\alpha}_t(w_\psi-h_\psi)(t)dt&=p\int_0^{+\infty}\int_0^{t}\frac{e^{-p(t-s)}(t-s)^{-\alpha}}{\Gamma(1-\alpha)}e^{-ps}w_\psi(s)ds-h_\psi p\int_0^\infty e^{-pt}\frac{t^{1-\alpha}}{\Gamma(2-\alpha)}dt\\
&=p^\alpha\int_0^{+\infty}e^{-ps}w_\psi(s)ds-p^{\alpha-1}h_\psi\\
&=p^\alpha\int_0^{T-\epsilon}e^{-ps}w_\psi(s)ds-p^{\alpha-1}h_\psi.\end{aligned}$$
Combining this two identities, for all $p\in(0,+\infty)$, we obtain
$$\begin{aligned} I^{1-\alpha}_t(w_\psi-h_\psi)(T)e^{-pT}+p\int_0^Te^{-pt}I^{1-\alpha}(w_\psi-h_\psi)(t)dt&=p\int_0^{+\infty}e^{-pt}I^{1-\alpha}_t(w_\psi-h_\psi)(t)dt\\
&=p^\alpha\int_0^{T-\epsilon}e^{-ps}w_\psi(s)ds-p^{\alpha-1}h_\psi.\end{aligned}$$
Recalling that $p\mapsto I^{1-\alpha}_t(w_\psi-h_\psi)(T)e^{-pT}+p\int_0^Te^{-pt}I^{1-\alpha}_t(w_\psi-h_\psi)(t)dt$ and $p\mapsto \int_0^{T-\epsilon}e^{-ps}w_\psi(s)ds$ are holomorphic in $\mathbb C$ and applying again isolated zero theorem, for all $p\in\mathbb C\setminus(-\infty,0]$, we find
\bel{tmb}I^{1-\alpha}_t(w_\psi-h_\psi)(T)e^{-pT}+p\int_0^Te^{-pt}I^{1-\alpha}_t(w_\psi-h_\psi)(t)dt=p^\alpha\int_0^{T-\epsilon}e^{-ps}w_\psi(s)ds-p^{\alpha-1}h_\psi,\ee
where $p^\alpha=e^{\alpha\log(p)}$, with $\log$ the complex logarithm defined on $\mathbb C\setminus(-\infty,0]$. Now choosing in \eqref{tmb}, $p=re^{i\theta}$, $r>0$, $\theta\in(-\pi,\pi)$, sending $\theta\to\pm\pi$ and taking the difference we get
$$\begin{aligned}0&=I^{1-\alpha}_t(w_\psi-h_\psi)(T)e^{rT}-r\int_0^Te^{rt}I^{1-\alpha}_t(w_\psi-h_\psi)(t)dt\\
&\ \ \ -\left(I^{1-\alpha}_t(w_\psi-h_\psi)(T)e^{rT}-r\int_0^Te^{rt}I^{1-\alpha}_t(w_\psi-h_\psi)(t)dt\right)\\
&=r^\alpha e^{i\alpha \pi}\int_0^{T-\epsilon}e^{rs}w_\psi(s)ds+ r^{\alpha-1}e^{i\alpha \pi}h_\psi -r^\alpha e^{-i\alpha \pi}\int_0^{T-\epsilon}e^{rs}w_\psi(s)ds-r^{\alpha-1}e^{-i\alpha \pi}h_\psi,\quad r>0.\end{aligned}$$
It follows that
$$2i\sin(\alpha\pi)\left(\int_0^{T-\epsilon}e^{rs}w_\psi(s)ds+r^{-1}h_\psi\right)=0,\quad r>0,$$
and, since $\alpha\pi\in(0,\pi)$, we deduce that 
$$\int_0^{T-\epsilon}e^{rs}w_\psi(s)ds=-h_\psi r^{-1},\quad r>0.$$
Recalling that 
$$\lim_{r\to0}\int_0^{T-\epsilon}e^{rs}w_\psi(s)ds=\int_0^{T-\epsilon}w_\psi(s)ds,$$
we deduce that $h_\psi=0$ and
$$\int_0^{+\infty}e^{-(-r)s}w_\psi(s)ds=\int_0^{T-\epsilon}e^{rs}w_\psi(s)ds=0,\quad r>0.$$
Combining this with  isolated zero theorem and the injectivity of the Laplace transform, we obtain $w_{\psi}\equiv0$ and $h_\psi=0$. Finally, recalling that $\psi\in L^2(X)$ is arbitrary chosen, we get $w(t,x)=h(x)=0$, $(t,x)\in(0,+\infty)\times X$, which proves that \eqref{tmab} holds true. This completes the proof of the theorem.
\end{proof}

%%%%%%%%%%%%%%%%%%%%%%%%%%%%%%%%%%%%%%%%%%%%%%%%%
%%%%%%%%%%%%%%%%%%%%%%%%%%%%%%%%%%%%%%%%%%%%%%%%%
%%%%%%%%%%%%%%%%%%%%%%%%%%%%%%%%%%%%%%%%%%%%%%%%%

\end{document}